\documentclass{amsart}

\usepackage{amsfonts,amssymb,mathrsfs}
\usepackage{graphicx}
\usepackage[all]{xy}
\usepackage{amscd}
\usepackage{color}
\usepackage{slashbox}
\usepackage{colortbl}

\usepackage{amsmath}

\usepackage[onehalfspacing]{setspace}
\usepackage{pdflscape}
\usepackage{colortbl}
\usepackage[table]{xcolor}
\newcolumntype{U}{>{\columncolor[gray]{0.8}}c}
\usepackage{tabularx,booktabs}
\usepackage{boxedminipage}
\usepackage{graphicx}
\usepackage{rotating}
\usepackage{booktabs}
\usepackage{longtable}
\usepackage{subfigure}
\usepackage{wrapfig}
\usepackage{lipsum}
\usepackage{multirow}
\usepackage{color}
\usepackage{fancyhdr}
\usepackage{amsmath}

\newtheorem{thm}{Theorem}[section]

\newtheorem{cor}[thm]{Corollary}
\newtheorem{lem}[thm]{Lemma}
\newtheorem{prop}[thm]{Proposition}

\theoremstyle{definition}
\newtheorem{definition}[thm]{Definition}

\newtheorem{rmk}[thm]{Remark}

\begin{document}

\baselineskip 12.70pt

\title{The same $n$-type structure of the suspension of the wedge products of the Eilenberg-MacLane spaces}

\thanks{
This research was supported by Basic Science Research Program through the
National Research Foundation of Korea (NRF) funded by the Ministry of
Education (NRF-2015R1D1A1A09057449).}
\thanks{
~~ Tel.: +82-63-270-3367.}

\author{Dae-Woong Lee}

\address{Department of Mathematics, and Institute of Pure and Applied Mathematics, Chonbuk National University,
567 Baekje-daero, Deokjin-gu, Jeonju-si, Jeollabuk-do 54896,
Republic of Korea} \email{dwlee@jbnu.ac.kr}


\begin{abstract}
For a connected CW-complex, we let $SNT(X)$ be the set of all homotopy types $[Y]$ such that the
Postnikov approximations $X^{(n)}$ and $Y^{(n)}$ of $X$ and $Y$, respectively, are homotopy equivalent for all positive integers $n$.
In 1992, McGibbon and M{\o}ller (\cite[page 287]{MM}) raised the following question: Is $SNT(\Sigma \mathbb C P^\infty) = *$ or not?
In this article, we give an answer to the more generalized version of this query: The set of all the same $n$-types of the suspended wedge sum of the Eilenberg-MacLane spaces of various types of both even and odd integers is the set which consists of only one element as a single homotopy type of itself.
\end{abstract}

\maketitle

\section{Introduction}

Let $Y$ be a connected CW-space and let $Y^{(n)}$ denote its $n$th Postnikov approximation up through dimension $n$. We recall that $Y^{(n)}$ can be obtained from $Y$ by attaching cells of dimension $n+2$ and higher cells in order to achieve a new space all of whose homotopy groups are zero in dimensions greater than $n$; that is, to kill off the generators of the homotopy groups of $Y$ in dimensions above $n$.
In this case, there exist maps $f_n : Y \rightarrow Y^{(n)}, n \geq 1$ such that ${f_n}_\sharp : \pi_i(Y) \rightarrow \pi_i(Y^{(n)})$ is an isomorphism for $i \leq n$ and $\pi_i (Y^{(n)}) = 0$ for all $i \geq n+1$. Furthermore, there exist maps $p_{n+1} : Y^{(n+1)} \rightarrow Y^{(n)}$ such that
$$
\xymatrix@C=10mm @R=1mm{
K(\pi_{n+1}(Y),n+1) \ar[r]^-{} &Y^{(n+1)} \ar[r]^-{p_{n+1}} &Y^{(n)}
}
$$
is a fibration.
We say that two connected CW-spaces $Y$ and $Z$ have the {\it same
$n$-type} if the $n$th Postnikov approximations $Y^{(n)}$ and
$Z^{(n)}$ of $Y$ and $Z$, respectively, are homotopy equivalent for each $n \geq 1$.

Back in the middle of the 20th century, the following question is mainly due to J. H. C. Whitehead (\cite{W1} and \cite{W2}): If two CW-complexes are of the same $n$-type for all $n$, are they necessarily of the same homotopy type?
It is well known that the answer to this query is yes if $Y$ is either finite dimensional or if $Y$ has only a finite number of nonzero homotopy groups. However, the answer is generally no! In 1957, Adams \cite{A} showed that if $K$ is a simply connected finite noncontractible CW-complex and if $Y = \prod_{n \geq 1} K^{(n)}$ with the direct limit topology, then $Y$ and $Y \times K$ have the same $n$-type for all $n$, but they are not of the same homotopy type. We note that the Adam's example does not have finite type. In 1966, Gray \cite{G} constructed one with finite type.

In this paper, we work on the pointed homotopy category so that we do not distinguish notationally between a
base point preserving continuous map and its homotopy class. As usual,
we denote $\mathbb Z_+$, $\mathbb Z$ and $\mathbb Q$ by the sets of all positive integers, integers and rational numbers, respectively.
And we will make use of the notations $\Sigma$ and $\Omega$ for the suspension functor and the loop functor in the pointed homotopy category, respectively.

Let $SNT(X)$ denote the set of all homotopy types $[Y]$ such that the Postnikov approximations $X^{(n)}$ and $Y^{(n)}$
of $X$ and $Y$, respectively, have the same homotopy types for all integers $n \geq 1$. This is a base point set with the base point $*=[X]$.
In 1976, Wilkerson \cite[Theorem I]{WI} showed that there is a bijection of pointed sets
$SNT(X) \; \approx \; {\varprojlim_n^1} \{ \text{Aut} (X^{(n)} ) \}$,
where $X$ is a connected $CW$-complex,
${\varprojlim_n^1} (-)$ is the first derived limit of groups in the
sense \cite[page 251]{BK} of Bousfield-Kan and $\text{Aut} X^{(n)}$ is the group of homotopy classes of homotopy self-equivalences of $X^{(n)}$.
Therefore, it can be seen that the torsion subgroups of a graded homotoy group $\pi_* (X^{(n)})$ can be ignored in the ${\varprojlim_n^1}$ computation if $X$ is a space of finite type.
For $m \geq 2$, an Eilenberg-MacLane spaces $K ({\mathbb Z}, m)$ of type $({\mathbb Z}, m)$ is infinite dimensional and it has a lot of torsion elements in homotopy groups for sufficiently large $m$ and it is unique up to homotopy.
Fortunately, it is well known that the set of all the same $n$-types for the $l$th suspension of the Eilenberg-MacLane space of type $({\mathbb Z}, 2a+1)$ is trivial for $l \geq 0$; that is,
$SNT(\Sigma^l K ( {\mathbb Z}, 2a+1)) = *$.
One of the reasons of this fact is that $\Sigma^l K ( {\mathbb Z}, 2a+1)$ has a
rational homotopy type of a single $n$-dimensional rationalized sphere $S^{n}_{\mathbb Q}$, where $n = l+2a+1$.
The even dimensional case, however, is much more complicated in
that $\Sigma K( {\mathbb Z}, 2a)$ has a rational homotopy type of a
bouquet of infinitely many rationalized spheres of dimensions $2a+1, 4a+1, \ldots, 2an+1,\ldots$.
So it is natural for us to ask in the case of even integers. First, the interesting case ($a =1$) is the following question posed by McGibbon and  M{\o}ller (\cite[page 287]{MM}): Is $SNT (\Sigma K( {\mathbb Z}, 2)) = *$ or not?
The positive answer to this question was given in \cite{L1}.
Secondly, the original question was how we know whether or not the set of all same $n$-type structures of the one suspension of the Eilenberg-MacLane space $K(\mathbb Z, 2a), a \geq 2$ of type $(\mathbb Z, 2a)$ is trivial in the same $n$-type point of view.
The answer to this question was also given in \cite{L2} saying that it is the one element set which consists of a single homotopy type of itself; see \cite{L3, L4} for the same $n$-type structures of the suspension of the smash products of those spaces or a wedge of $K(\mathbb Z, 2a)$s.

More generally, what will happen in the case of the suspended wedge sum of the Eilenberg-MacLane spaces of some types of both even and odd integers?
The wedge product can be considered as the coproduct in the homotopy category of pointed spaces.
After taking suspensions or wedge products of the Eilenberg-MacLane spaces $K( {\mathbb Z}, 2a)$
and $K ( {\mathbb Z}, 2a+1)$ for $a \geq 1$ as the infinite loop
spaces, they become interesting but much more intractable for us to deal with. Therefore, it is worth mentioning what it is in the same $n$-type point of view up to homotopy.
Despite the salient results on this topics, little is known about the same $n$-type structures of the wedge product of those spaces. The main purpose of this paper is to give an answer to this query as the
generalized version of the McGibbon-M{\o}ller's original question on the same $n$-types of the suspension of the infinite complex projective space: Let $X = K( {\mathbb Z}, 2a) \vee \bigvee_{j=1}^{\infty} K( {\mathbb Z}, 2aj+1)$ be the wedge product of the
Eilenberg-MacLane spaces, where `$a$' is a positive integer. Then $SNT(\Sigma X) = \{ [\Sigma X] \}$.

As the dual notion of Hopf spaces with multiplications, co-H-spaces with comultiplications play a pivotal
role in homotopy theory. We note that a non-contractible co-H-space is the space of
Lusternik-Schnirelmann category $2$ (see \cite[Chapter X]{W} and \cite{SC}); that is, $\rm{cat} (X) = 2$ which is bigger than or equal to the $R$-cup length of the cohomology of this space with coefficients in a commutative ring $R$. One of the most important classes of co-H-spaces consists of all $n$-spheres for $n\ge 1$, a wedge of spheres (or co-H-spaces), and the suspensions of a pointed space which we will deal with in this article.

The paper is organized as follows: In Section 2, we construct self-maps, (pure and hybrid) iterated commutators and homotopy self-equivalences by using the suspension and loop structures, and describe the fundamental results of those maps. In Section 3, we consider the iterated Samelson (or Whitehead) products  which are rationally nonzero homotopy classes in the homotopy groups module torsions, and show that there exist (pure and hybrid) iterated commutators on the suspension and loop structures whose phenomena are exactly the same as the types of the iterated Samelson (or Whitehead) products as basic ingredients on this paper. Finally, In Section 4, we describe the main theorem of this paper and make use of all the results in Sections 2 and 3 for the proof of the theorem.

\medskip

\section{Self-maps and iterated commutators}\label{section2}

In this section, we consider the self-maps of the suspended wedge sum of the Eilenberg-MacLane spaces and the iterated commutators of those self-maps in order to construct homotopy self-equivalences and to find out various properties of those self-maps. We fix the following notations that are used throughout this paper:
\begin{itemize}
\item Let $X: = K( {\mathbb Z}, 2a) \vee \bigvee_{j=1}^{\infty} K( {\mathbb Z}, 2aj+1)$ denote a CW-space on the wedge product of the
      Eilenberg-MacLane spaces, where `$a$' is a positive integer.
\item $X_n$ means the $n$-skeleton of $X$.
\item $C : \Omega \Sigma X \wedge  \Omega \Sigma X \rightarrow \Omega \Sigma X$ is a commutator map; that is,
$$
C(f \wedge g) = f \cdot g \cdot f^{-1} \cdot g^{-1},
$$
where the multiplication is originated from the loop structure on $\Omega \Sigma X$ ($\simeq JX$ as an $X$-cellular space), and $f^{-1}$ and $g^{-1}$ are the loop inverses of $f$ and $g$, respectively.
\end{itemize}

We note that $X$ has a CW-decomposition as follows:
$$
\begin{array}{lll}
X &= (S^{2a} \cup T^1_1 \cup_{\alpha_1} e^{4a}  \cup T^1_2 \cup_{\alpha_2} e^{6a}  \cup \cdots
\cup T^1_{i-1} \cup_{\alpha_{i-1}}e^{2ai} \cup T^1_{i} \cup_{\alpha_i}e^{2a(i+1)} \cup \cdots) \\
  &\qquad \vee (S^{2a+1} \cup T_1^2 ) \vee (S^{4a+1} \cup T_2^2 ) \vee \cdots \vee (S^{2aj+1} \cup T_j^2 ) \vee
  (S^{2a(j+1)+1} \cup T_{j+1}^2 ) \vee \cdots .
\end{array}
$$
Here
\begin{itemize}
\item [{\rm (1)}] $\alpha_i : S^{2a(i+1)-1} \rightarrow X_{2a(i+1)-1}$
is an attaching map for $i=1,2,3, \ldots$;
\item [{\rm (2)}] $T^1_i$ and $T_j^2$ denote the other cells for torsion subgroups or the Moore spaces whose homology groups are finite for
      $i, j = 1,2,3, \ldots$; and
\item [{\rm (3)}] $e^{2ai}$ is the $2ai$-cell for $i = 2,3, 4 \ldots$.
\end{itemize}

In order to define self-maps of $\Sigma X$, we first consider the base point preserving maps as follows:
\begin{itemize}
\item [$\diamond$] Let $\hat \varphi_1 : X \rightarrow \Omega \Sigma X$ be the canonical inclusion (or the James map).
\item [$\diamond$] We consider cofibrations
\begin{equation}\tag{2.1}
\xymatrix@C=3.0pc{
X_{2ai-1} \,\, \ar@{^{(}->}[r]^-{\iota_i} &X \ar[r]^-{p_{\iota_i} } & X/X_{2ai-1}
}
\end{equation}
and
\begin{equation}\tag{2.2}
\xymatrix@C=3.0pc{
X_{2aj} \,\, \ar@{^{(}->}[r]^-{\zeta_j} &X \ar[r]^-{p_{\zeta_j}} & X/X_{2aj},
}
\end{equation}
where
 \begin{enumerate}
    \item [{\rm (1)}] $\iota_i$ and $\zeta_j$ are inclusions; and
    \item [{\rm (2)}] $p_{\iota_i}$ and $p_{\zeta_j}$ are projections for $i = 2,3,4,\ldots $ and $j = 1,2,3,\ldots$ (we note that $X_{2a-1} = *$).
 \end{enumerate}
\item [$\diamond$] We also consider the following exact sequences
\begin{equation}\tag{2.1-a}
\xymatrix{
[ X/X_{2ai-1} ,\Omega \Sigma X] \ar[r]^-{p_{\iota_i}^{\sharp}} &[X, \Omega \Sigma X] \ar[r]^-{\iota_i^{\sharp}}
&[ X_{2ai-1}, \Omega \Sigma X]
}
\end{equation}
and
\begin{equation}\tag{2.2-a}
\xymatrix{
[ X/X_{2aj} ,\Omega \Sigma X] \ar[r]^-{p_{\zeta_j}^{\sharp}} &[X, \Omega \Sigma X] \ar[r]^-{\zeta_j^{\sharp}}
&[ X_{2aj}, \Omega \Sigma X]
}
\end{equation}
induced by the cofibrations above for $i = 2,3,4,\ldots $ and $j = 1,2,3,\ldots$.
\item [$\diamond$] We now take essential maps
$$
\xymatrix@C=1.5pc{
\hat \varphi_{i} : X \ar[r]^{} &\Omega \Sigma X }
$$
and
$$
\xymatrix@C=1.5pc{
\hat \psi_{j} : X \ar[r]^{} &\Omega \Sigma X }
$$
in the groups ${\rm ker}(\iota_i^{\sharp}) \subset [X, \Omega \Sigma X]$
and
${\rm ker}(\zeta_j^{\sharp}) \subset [X, \Omega \Sigma X]$, respectively for $i = 2,3,4,\ldots $ and $j = 1,2,3,\ldots$.
\end{itemize}

From the constructions of $\hat \varphi_{i} : X \rightarrow \Omega \Sigma X$ and $\hat \psi_{j} : X \rightarrow \Omega \Sigma X $ above, we note that there exist homotopy classes $\bar \varphi_{i}$ and $\bar \psi_{i}$ in the groups
$[X/X_{2ai-1} , \Omega \Sigma X]$ and $[X/X_{2aj-1}, \Omega \Sigma X]$, respectively, such that following diagrams
\begin{equation}\tag{2.3}
\xymatrix@C=8mm @R=8mm{
X_{2ai-1} \ar@{^{(}->}[r]^-{\iota_i} \ar[dr]_-{\hat \varphi_{i}|_{X_{2ai-1}}} &X  \ar[r]^-{p_{\iota_i}} \ar[d]_-{\hat \varphi_{i}} &X/X_{2ai-1}
\ar[dl]^-{\bar \varphi_{i}}  \\
   &\Omega \Sigma X
}
\end{equation}
and
\begin{equation}\tag{2.4}
\xymatrix@C=8mm @R=8mm{
X_{2aj} \ar@{^{(}->}[r]^-{\zeta_j} \ar[dr]_-{\hat \psi_{j}|_{X_{2aj}}} &X  \ar[r]^-{p_{\zeta_j}} \ar[d]_-{\hat \psi_{j}} &X/X_{2aj}
\ar[dl]^-{\bar \psi_{j}}  \\
   &\Omega \Sigma X
}
\end{equation}
are commutative up to homotopy for $i = 2,3,4,\ldots $ and $j = 1,2,3,\ldots$.

The following tables (Table 1 and Table 2) indicate that the restrictions $\hat \varphi_{i}|_{X_{2ai-1}} : X_{2ai-1} \rightarrow \Omega \Sigma X$ and $\hat \psi_{j}|_{X_{2aj}} : X_{2aj} \rightarrow \Omega \Sigma X$ of maps $\hat \varphi_{i} : X \rightarrow \Omega \Sigma X$ and
$\hat \psi_{j} : X \rightarrow \Omega \Sigma X$, respectively, to the corresponding skeletons are inessential for
$i = 1,2,3,\ldots $ and $j = 1,2,3,\ldots$.
\renewcommand{\tabcolsep}{10pt}
\renewcommand{\arraystretch}{1.5}
\begin{table}[h!]
  \begin{center}
    \caption{Maps $\hat \varphi_{i} : X \rightarrow \Omega \Sigma X$ and $\hat \varphi_{i}|_{X_{2ai-1}} : X_{2ai-1} \rightarrow \Omega \Sigma X$ for $i = 1,2,3,\ldots $}
    \label{tab:table1}
\begin{tabularx}{0.980\linewidth}
{|>{\columncolor{lightgray}}c|X|X|X|X|X|X|}
\hline
Maps   &    $\hat \varphi_{1}$   &   $\hat \varphi_{2}$   &    $\hat \varphi_{3}$ & $\cdots$ &$\hat \varphi_{i}$ & $\cdots$ \\
	\hline
Skeletons & $X_{2a-1}$ & $X_{4a-1}$  & $X_{6a-1}$  &$\cdots$ &$X_{2ai-1}$ &$\cdots$\\
\hline
Inessential maps & $\hat \varphi_{1}|_{X_{2a-1}}$ & $\hat \varphi_{2}|_{X_{4a-1}}$  & $\hat \varphi_{3}|_{X_{6a-1}}$  &$\cdots$ &$\hat \varphi_{i}|_{X_{2ai-1}}$ &$\cdots$\\
	\hline
\end{tabularx}
\end{center}
\end{table}

\renewcommand{\tabcolsep}{10pt}
\renewcommand{\arraystretch}{1.5}
\begin{table}[h!]
  \begin{center}
    \caption{Maps $\hat \psi_{j} : X \rightarrow \Omega \Sigma X$ and $\hat \psi_{j}|_{X_{2aj}} : X_{2aj} \rightarrow \Omega \Sigma X$ for $j = 1,2,3,\ldots $}
    \label{tab:table1}
\begin{tabularx}{0.980\linewidth}
{|>{\columncolor{lightgray}}c|X|X|X|X|X|X|}
\hline
Maps   &    $\hat \psi_{1}$   &   $\hat \psi_{2}$   &    $\hat \psi_{3}$ & $\cdots$ &$\hat \psi_{j}$ & $\cdots$ \\
	\hline
Skeletons & $X_{2a}$ & $X_{4a}$  & $X_{6a}$  &$\cdots$ &$X_{2aj}$ &$\cdots$\\
\hline
Inessential maps & $\hat \psi_{1}|_{X_{2a}}$ & $\hat \psi_{2}|_{X_{4a}}$  & $\hat \psi_{3}|_{X_{6a}}$  &$\cdots$ &$\hat \psi_{j}|_{X_{2aj}}$ &$\cdots$\\
	\hline
\end{tabularx}
\end{center}
\end{table}

Let $Y = S^{n_1} \vee S^{n_2 } \vee \cdots \vee S^{n_k}$ and let $\rho_j: S^{n_j} \longrightarrow Y$ be the canonical inclusion
for $j=1,2, \ldots,k$.
We define and order the {\it basic Whitehead products} on $Y$ as
follows: Basic Whitehead products of weight 1 are $\rho_1,\rho_2, \ldots ,
\rho_k$ which are ordered by $\rho_1 <\rho_2 <\cdots < \rho_k$.
Assume that the basic Whitehead products of weight $< n$ have already been defined and
ordered so that if $r<s <n$, any basic Whitehead product of weight $r$ is less
than all basic Whitehead products of weight $s$. By induction, we define a basic Whitehead
product of weight $n$ by a basic Whitehead product $[A,B]$, where $A$ is a
basic Whitehead product of weight $\alpha$ and $B$ is a basic Whitehead product of weight
$\beta$, $\alpha + \beta =n$, $A<B$, and if $B$ is a basic Whitehead product $[C,D]$ of
basic Whitehead products $C$ and $D$, then $C\le A$. The basic Whitehead products of
weight $n$ are ordered arbitrarily among themselves and are
greater than any basic Whitehead product of weight $<n$.  Suppose $\rho_j$
occurs $l_j$ times, $l_j \ge 1$ in the basic Whitehead product
$w_v$.  Then the {\it height} $h_v$ of the basic Whitehead product $w_v$ is
$\sum l_j (n_j-1)+1$.

We describe the beautiful Hilton's formula \cite{H} in terms of the basic Whitehead products as follows:

\begin{thm} \label{H} Let $w_1, w_2, \ldots , w_v, \ldots $ be the basic Whitehead products of
$$
Y = S^{n_1} \vee S^{n_2} \vee \cdots \vee S^{n_k}
$$
with the height $h_v$ of $w_v$, where $v =1, 2, 3,\ldots$. Then we obtain
$$
\pi_m (Y) \cong \bigoplus _{v=1}^{\infty} \pi_m(S^{h_v})
$$
for every $m$. The isomorphism $\theta : \bigoplus _{v=1}^{\infty} \pi_m (S^{h_v}) \to \pi_m (Y)$ is defined by
$$
\theta |_{\pi_m(S^{h_v})} = w_{{v}_*} :  \pi_m(S^{h_v}) \longrightarrow
\pi_m (Y),
$$
where $w_v$ is the basic Whitehead product.
\end{thm}

From a CW-decomposition of $X: = K( {\mathbb Z}, 2a) \vee \bigvee_{j=1}^{\infty} K( {\mathbb Z}, 2aj+1)$, we note that
$$
X/X_{2ai-1} = S^{2ai} \cup \text{higher cells whose dimensions are bigger than}~ 2ai
$$
and
$$
X/X_{2aj} = S^{2aj+1} \cup \text{higher cells whose dimensions are bigger than}~ 2aj +1
$$
for $i, j = 1,2,3,\ldots$.


\begin{definition} \label{def2.2}
We define the rationally nonzero homotopy classes $\hat x_{i}$ and $\hat x_{j}$ of the homotopy groups modulo torsion subgroups $\pi_{2ai}(\Omega \Sigma X) /\rm{torsion}$ and $\pi_{2aj+1}(\Omega \Sigma X) /\rm{torsion}$ by
$$
\xymatrix@C=7mm @R=5mm{
\hat x_{1} = \hat \varphi_{1} |_{S^{2a}} : S^{2a} \ar[r] &\Omega \Sigma X;
}
$$
$$
\xymatrix@C=7mm @R=5mm{
\hat x_{i} = \bar \varphi_{i} |_{S^{2ai}} : S^{2ai} \ar[r] &\Omega \Sigma X;
}
$$
and
$$
\xymatrix@C=7mm @R=5mm{
\hat y_{j} = \bar \psi_{j} |_{S^{2aj+1}} : S^{2aj+1} \ar[r] &\Omega \Sigma X,
}
$$
respectively, for $i = 2,3,4,\ldots $ and $j = 1,2,3,\ldots$.
\end{definition}

We recall that the suspension $\Sigma$ gives rise to a covariant functor from the pointed homotopy category of pointed spaces to itself. An important and fundamental property of this functor is that it is a left adjoint to the loop functor $\Omega$ taking a pointed space $X$ to its loop space $\Omega X$; that is $(\Sigma, \Omega)$ is an adjoint pair of covariant functors on the pointed homotopy category. Therefore, we can define the following:

\begin{definition}
We define the self-maps $\varphi_{i} : \Sigma X \rightarrow \Sigma X$ and $\psi_{j} : \Sigma X \rightarrow \Sigma X$
as the adjointness of $\hat \varphi_{i} : X \rightarrow \Omega \Sigma X$ and $\hat \psi_{j} : X \rightarrow \Omega \Sigma X$, respectively. Similarly, we define maps $x_{i} : S^{2ai+1} \rightarrow \Sigma X$ and $y_{j} : S^{2aj+2} \rightarrow \Sigma X$ as the adjointness of  $\hat x_{i} : S^{2ai} \rightarrow \Omega \Sigma X$ and $\hat y_{j} : S^{2aj+1} \rightarrow \Omega \Sigma X$, respectively for $i, j = 1,2,3,\ldots$.
\end{definition}

We now give an order of the basic Whitehead products of weight $1$ on the rationally nonzero homotopy classes $x_{i}$ and $x_{j}$ of the graded homotopy groups modulo torsion subgroups, $\pi_* (\Sigma X)/\rm{torsion}$, as follows:
$$
x_1 < y_1 <x_2 < y_2 < \ldots < x_i < y_i < \ldots
$$
for $i = 1,2,3, \ldots$. We then define and order the basic Whitehead products of weight $n$ on $\pi_* (\Sigma X)/\rm{torsion}$ as just mentioned above.

Based on the above self-maps $\varphi_{i} : \Sigma X \rightarrow \Sigma X$ and $\psi_{j} : \Sigma X \rightarrow \Sigma X$ for $i, j = 1,2,3,\ldots$, we now construct new self-maps as follows.

\begin{definition}
We define the {\it commutator} of self-maps $f_{s_{t}} : \Sigma X \rightarrow \Sigma X, t =1,2$ on $\Sigma X$ by
$$
[f_{s_1}, f_{s_2}]_c = f_{s_1} + f_{s_2} -f_{s_1} -f_{s_2} : \Sigma X \longrightarrow \Sigma X
$$
where $f_{s_1} = \varphi_{s_1}$ or $\psi_{s_1}$ and  $f_{s_2} = \varphi_{s_2}$ or $\psi_{s_2}$ for $s_1, s_2 = 1,2,3,\ldots$, and the addition and subtraction are derived from the suspension $\Sigma X$ of $X$.
Inductively, we also define the {\it iterated commutator}
$$
[f_{s_l}, [f_{s_{l-1}},\ldots,[f_{s_{1}}, f_{s_2}]_c \ldots]_c]_c : \Sigma X \longrightarrow \Sigma X
$$
of self-maps $f_{s_{t}} : \Sigma X \rightarrow \Sigma X, t =1,2,\ldots,l$ on the suspension structure again, where $f_{s_t} = \varphi_{s_t}$ or $\psi_{s_t}$ for $s_t = 1,2,3,\ldots$.
\end{definition}

\begin{definition}
The iterated commutator
$$
[f_{s_l}, [f_{s_{l-1}},\ldots,[f_{s_{1}}, f_{s_2}]_c \ldots]_c]_c : \Sigma X \longrightarrow \Sigma X
$$
is said to be be {\it pure} if each map $f_{s_t}$, $t=1,2,\ldots,l$, on the iterated commutator appears only as either $\varphi_{s_i}$ or $\psi_{s_j}$ for $i,j =1,2,\ldots,l$.
It is said to be {\it hybrid} if it contains at least one $\varphi_{s_i}$ and at least one $\psi_{s_j}$ on the iterated commutator for $s_i , s_j = 1,2,3,\ldots$ and $i,j \in \{1,2,\ldots,l\}$.
\end{definition}

Since the set $[\Sigma X, \Sigma X]$ of homotopy classes of pointed maps from $\Sigma X$ to itself has just a group structure (not necessarily abelian), the iterated commutators of self-maps above are not necessarily inessential and thus they are worth considering; see \cite{Mo} for the infinite complex projective space. To top it all off, the iterated commutators do make sense because there are infinitely many non-zero cohomology cup products in $X$ so that
the Lusternik-Schnirelmann category is infinite. Moreover, it is well known that the $n$-fold iterated commutator is of finite order if and only if all $n$-fold cohomology cup products of rational cohomology classes of a space are all zero; see \cite{M} and \cite[Theorem 5]{MC}.

\begin{definition}\label{self-map}
Let $1 : \Sigma X \rightarrow \Sigma X$ be the identity map on $\Sigma X$. Then by using the notations above, we define the self-maps
$$
1 + [f_{s_l}, [f_{s_{l-1}},\ldots,[f_{s_{1}}, f_{s_2}]_c \ldots]_c]_c : \Sigma X \longrightarrow \Sigma X ,
$$
where the addition $+$ is the suspension one.
\end{definition}

\begin{rmk}
We note that the self-maps on $\Sigma X$ in Definition \ref{self-map}
are indeed the homotopy self-equivalences of $\Sigma X$ by the Whitehead theorem.
\end{rmk}

We refer to the Arkowitz's works \cite{M1, M2} and the Rutter's book \cite{R} for a survey of the vast literature about homotopy self-equivalences and related topics; see also \cite{ML0}, \cite{ML} and \cite{ML1}.

Let $X_{\mathbb Q}$ be the rationalization (i.e., $0$-localization) of $X$ \cite{HMR}. Then it can be seen that the total rational homotopy group
$$
\hat {\mathcal{L}} = (\pi_{*} (\Omega \Sigma X_{\mathbb{Q}}), < ~,~>) = (\pi_{*} (\Omega \Sigma X ) \otimes \mathbb{Q}, < ~,~>)
$$
of $\Omega \Sigma X_{\mathbb{Q}}$ becomes a graded Lie algebra over the field $Bbb Q$ of rational numbers with the Samelson products $< ~,~>$ which is called the {\it rational homotopy Lie algebra} of $\Sigma X$.
Similarly, we denote $\hat {\mathcal L}_{\leq 2an}$ by the subalgebra of $\hat {\mathcal{L}}$ generated by all free algebra generators whose degrees are less than or equal to $2an$; that is,
$$
\hat {\mathcal{L}}_{\leq 2an} =  (\pi_{\leq 2an} (\Omega \Sigma X_{\mathbb{Q}}), < ~,~>) = (\pi_{\leq 2an} (\Omega \Sigma X ) \otimes \mathbb Q, < ~,~>)
$$
with rational homotopy generators $\hat \chi_1 , \hat \eta_1, \hat \chi_2 , \hat \eta_2, \ldots, \hat \eta_{i-1}, \hat \chi_i , \hat \eta_i, \ldots, \hat \chi_{n-1}, \hat \eta_{n-1}, \hat \chi_n$, where
$\hat \chi_i : S^{2ai}\rightarrow \Omega \Sigma X_{\mathbb Q}$ and $\hat \eta_j : S^{2aj+1}\rightarrow \Omega \Sigma X_{\mathbb Q}$
are the compositions $r \circ \hat x_i$ and  $r \circ \hat y_j$ of the rationally nonzero indecomposable elements
$\hat x_i : S^{2ai} \rightarrow \Omega \Sigma X$ of $\pi_{2ai} (\Omega \Sigma X)/{\rm torsion}$
and
$\hat y_j : S^{2aj+1} \rightarrow \Omega \Sigma X$
of $\pi_{2aj+1} (\Omega \Sigma X)/{\rm torsion}$, respectively, for $i= 1,2,\ldots, n$ and $j= 1,2,\ldots, n-1$,
with the topological rationalization map $r : \Omega \Sigma X \rightarrow \Omega \Sigma X_{\mathbb Q}$. The adjoint pair of functors $(\Sigma, \Omega)$ on the pointed homotopy category makes us think about the following graded quasi-Lie algebra
$$
\mathcal{L}_{\leq 2an+1} =  (\pi_{\leq 2an+1} (\Sigma X ) \otimes \mathbb Q, [ ~,~])
$$
which is called the {\it Whitehead algebra} with the rational homotopy generators $\chi_1 , \eta_1, \chi_2 ,  \eta_2, \ldots,$  $\eta_{i-1}, \chi_i , \eta_i, \ldots, \chi_{n-1},  \eta_{n-1}, \chi_n$ and the Whitehead products $[~ , ~]$ as brackets, where
$\chi_i : S^{2ai+1}\rightarrow \Sigma X_{\mathbb Q}$
and
$\eta_j : S^{2aj+2}\rightarrow \Sigma X_{\mathbb Q}$
are obtained by the adjointness of the rationally nonzero indecomposable elements $\hat \chi_i$ and $\hat \eta_j$ above, respectively,
for $i= 1,2,\ldots, n$ and $j= 1,2,\ldots, n-1$.

Let ${\rm Aut}(Y)$ be the group of pointed homotopy classes of all homotopy self-equivalences of a connected CW-space $Y$.
Then, in rational homotopy theory, the automorphism of the Whitehead algebras
$$
\kappa : \mathcal{L}_{\leq 2an+1} \longrightarrow \mathcal{L}_{\leq 2an+1},
$$
defined by $\kappa (\chi_i ) = \chi_i + {\rm decomposables}$ and $\kappa (\eta_j ) = \eta_j + {\rm decomposables}$
are corresponding to our homotopy self-equivalences
$$
1 + [f_{s_l}, [f_{s_{l-1}},\ldots,[f_{s_{1}}, f_{s_2}]_c \ldots]_c]_c \in  {\rm Aut}(\Sigma X),
$$
inducing the identity automorphism $1_{H_* (\Sigma X; \mathbb Q)}$ on homology with rational coefficients in the
group of automorphisms ${\rm Aut} (H_* (\Sigma X; \mathbb Q))$,
where $f_{s_{t}}$ are self-maps of $\Sigma X$ with $f_{s_t} = \varphi_{s_t}$ or $\psi_{s_t}, t= 1,2,\ldots l$ and $s_t = 1,2,3,\ldots$.

\begin{rmk}
In general, if $A$ is a finite $CW$-complex or locally compact, and if $f,g : A \rightarrow \Omega B$ are the pointed preserving continuous maps satisfying $f \vert_{A_s} \simeq c_{b_0}$ and $g \vert_{A_t} \simeq c_{b_0}$, where $c_{b_0}$ is the constant loop at $b_0 \in B$,
then it can be shown that the restriction of the commutator $[f, g] : A \rightarrow \Omega B$ to the $n$-skeleton of $A$ is inessential, where
$n = s+t+1$.
\end{rmk}

Thus, by the constructions of self-maps on $\Sigma X$ as adjoint maps (see Tables 1 and 2) and by using induction, we have the following (see also \cite[Lemma 3.4]{L3}):

\begin{rmk} \label{Res} {\rm(a)} Let the iterated commutator
$F: \Sigma X \rightarrow \Sigma X$ be pure. If
$$
F_1 = [\varphi_{s_{l_1}}, [\varphi_{s_{l_1 -1}},\ldots,[\varphi_{s_{1}}, \varphi_{s_2}]_c \ldots]_c]_c : \Sigma X \rightarrow \Sigma X,
$$
then
$$
F_1 |_{\Sigma X_{k_1}} \simeq * : \Sigma X_{k_1} \hookrightarrow \Sigma X
$$
for $k_1 = 2a(s_1+s_2+\cdots+s_{l_1})$.
If
$$
F_2 = [\psi_{s_{l_2}}, [\psi_{s_{l_2 -1}},\ldots,[\psi_{s_{1}}, \psi_{s_2}]_c \ldots]_c]_c : \Sigma X \rightarrow \Sigma X,
$$
then
$$
F_2 |_{\Sigma X_{k_2}} \simeq * : \Sigma X_{k_2} \hookrightarrow \Sigma X
$$
for $k_2 = 2a(s_1+s_2+\cdots+s_{l_2})+l_2$.

{\rm(b)} Let the iterated commutator
$$
F = [f_{s_l}, [f_{s_{l-1}},\ldots,[f_{s_{1}}, f_{s_2}]_c \ldots]_c]_c : \Sigma X \rightarrow \Sigma X
$$
be hybrid which contains both the $t$-times of $\varphi_{s_{i}}$ and the $(l-t)$-times of $\psi_{s_{j}}$, where $1 \leq t \leq l-1$. Then
$$
F|_{\Sigma X_{k}} \simeq * : \Sigma X_{k} \hookrightarrow \Sigma X
$$
for $k = 2a(s_1+s_2+\cdots+s_l)+l-t$.
\end{rmk}

Let
$$
i_1 : Y \rightarrow Y \times Y
$$
and
$$
i_2 : Y \rightarrow Y \times Y
$$
be the first and second inclusion maps; that is, $i_1(y) = (y, y_0)$ and $i_2(y) = (y_0, y)$. Then a {\it primitive element} is an element $p \in H_*(Y)$ that satisfies
$$
\Delta_* (p) = {i_1}_* (p) + {i_2}_* (p) = p \otimes 1 + 1 \otimes p,
$$
where $\Delta : Y \rightarrow Y \times Y$ is the diagonal map. It is very easy to show that each element of $H_n(S^n)$ is primitive for $n \geq 1$.

\begin{prop} The images of $\hat \varphi_{i_*} : H_{2ai} (X) \rightarrow H_{2ai} (\Omega \Sigma X)$ and $\hat \psi_{j_*} : H_{2aj+1} (X) \rightarrow H_{2aj+1} (\Omega \Sigma X)$ are primitive for each $i, j =1,2,3, \ldots$.
\end{prop}

\begin{proof}
In general, we first note that if $f : X \rightarrow Y$ is continuous map and if $x$ is a primitive element of $H_* (X; I)$ with coefficients in a principal ideal domain $I$, then from the commutative diagram
$$
\xymatrix@C=10mm @R=8mm{
H_* (X; I) \ar[r]^-{f_*} \ar[d]_-{\Delta_*}  &H_* (Y; I)  \ar[d]^-{\Delta_*} \\
H_* (X \times X; I) \ar[r]^-{(f \times f)_*} &H_* (Y \times Y; I),
}
$$
we have
$$
\begin{array}{lll}
\Delta_* \circ (f_* (x)) &= (f \times f)_* \circ \Delta_* (x) \\
&= (f \times f)_* ({i_1}_* (x) + {i_2}_* (x)) \\
&= {i_1}_* (f_* (x)) + {i_2}_* (f_* (x)), \\
\end{array}
$$
where $i_1 , i_2 : W \rightarrow W \times W$ is the first and second inclusions on $W = X$ or $Y$;
that is, the homomorphic image of the primitive elements is also primitive.

The constructions of the maps
$$
\bar \varphi_i : X/X_{2ai-1} \rightarrow \Omega \Sigma X ~({\rm resp.}~ \bar \psi_j : X/X_{2aj} \rightarrow \Omega \Sigma X)
$$
assert that the maps
$$
\hat \varphi_i : X \rightarrow \Omega \Sigma X ~({\rm resp.}~ \hat \psi_j : X \rightarrow \Omega \Sigma X)
$$
can be factored through the projections
$$
p_{\iota_i} : X \rightarrow X/X_{2ai-1} ~({\rm resp.}~ p_{\zeta_j} : X \rightarrow X/X_{2aj})
$$
and
$$
\bar \varphi_i : X/X_{2ai-1} \rightarrow \Omega \Sigma X ~({\rm resp.}~ \bar \psi_j : X/X_{2aj} \rightarrow \Omega \Sigma X)
$$
up to homotopy for each $i =2,3,4,\ldots$ and $j =1,2,3,\ldots$.
For $i=1$, we also note that the following diagram
$$
\xymatrix@C=8mm @R=8mm{
    X \ar[rr]^-{\hat \varphi_1 }  \ar[dr]_-{1_X}  && \Omega \Sigma X &\\
     &X = X/X_{2a-1} \ar[ur]_-{\hat \varphi_1}
     }
$$
is strictly commutative due to the fact that $X_{2a-1} = *$. Since $X/X_{2ai-1}$ is $(2ai-1)$-connected, by the Hurewicz isomorphism, we can see that every homology class of $H_{2ai} (X/X_{2ai-1})$ is spherical.
Since every spherical homology class is primitive, it is primitive. It can be seen that the image of
$$
\hat \varphi_{i_*} : H_{2ai}(X) \rightarrow H_{2ai}(\Omega \Sigma X)
$$
lies in the set of all primitive homology classes $PH_{2ai}(\Omega \Sigma X)$ in $H_{2ai}(\Omega \Sigma X)$ for each $i=1,2,3,\ldots$.
The argument for the other case is completely similar.
\end{proof}

The following is the generalized version of \cite[Lemma 3.2]{L1}.

\begin{prop} \label{Re}
Let $x_m$ be a rationally nonzero homotopy class of $\pi_{2am+1}(\Sigma X)$. Then
$$
{(1 + [\varphi_{s_{l_1}}, [\varphi_{s_{l_1 -1}},\ldots,[\varphi_{s_{1}}, \varphi_{s_2}]_c \ldots]_c]_c})_{\sharp} (x_m ) \,
=\, x_m + {[\varphi_{s_{l_1}}, [\varphi_{s_{l_1 -1}},\ldots,[\varphi_{s_{1}}, \varphi_{s_2}]_c \ldots]_c]_c}_{\sharp} (x_m),
$$
where $m = s_1+s_2+\cdots+s_{l_1}$, and the first addition is the one of suspension structure on $\Sigma X$, while the second addition refers to the one of homotopy groups.
\end{prop}

\begin{proof}
Let $F_1 = [\varphi_{s_{l_1}}, [\varphi_{s_{l_1 -1}},\ldots,[\varphi_{s_{1}}, \varphi_{s_2}]_c \ldots]_c]_c : \Sigma X \rightarrow \Sigma X$.
Then we first show that the following diagram
$$
\xymatrix@C=10mm @R=8mm{
S^{2am+1} \ar[r]^-{x_m} \ar[d]^{\nu} &\Sigma X  \ar[r]^{1 + F_1} &\Sigma X\\
S^{2am+1} \vee S^{2am+1} \ar[r]^-{x_m \vee x_m}  &\Sigma X \vee \Sigma X \ar[ur]_-{(1, F_1)}
}
$$
is commutative up to homotopy.
Here
\begin{enumerate}
\item [{\rm (1)}] $\nu : S^{2am+1} \rightarrow S^{2am+1} \vee S^{2am+1}$ is the (homotopically unique) standard comultiplication of the sphere;
\item [{\rm (2)}]  $1 + F_1 : \Sigma X \rightarrow \Sigma X$ is defined by the homotopy class of the composite
    \begin{equation}\tag{2.5}
    \xymatrix@C=10mm @R=10mm{
    \Sigma X \ar[r]^-{\nu} &\Sigma X \vee \Sigma X \ar[r]^-{1 \vee F_1} \ar@/_1.5pc/[rr]_{(1,F_1)}
     &\Sigma X \vee \Sigma X \ar[r]^{\nabla} &\Sigma X, }
   \end{equation}
where $\nu : \Sigma X \rightarrow \Sigma X \vee \Sigma X$ is the standard suspension comultiplication of $\Sigma X$ (which we also denote as $\nu$), and $\nabla : \Sigma X \vee \Sigma \rightarrow \Sigma X$ is the folding map; and
\item [{\rm (3)}] $(1 , F_1 ) : \Sigma X \vee \Sigma X \rightarrow \Sigma X$ is defined by the homotopy class of the composition of the latter two maps of (2.5).
\end{enumerate}
As usual, we let $\Sigma X_{2am}$ be the $2am$-skeleton of $\Sigma X$. Then from the construction of the self-maps $\varphi_i$ and $\psi_j$ of $\Sigma X$ for $i,j =1,2,3,\ldots$, we now have the following commutative diagram
$$
\xymatrix@C=10mm @R=8mm{
    \Sigma X \ar[rr]^-{F_1}  \ar[dr]_-{\Sigma p_{\zeta_m}}  &&\Sigma X &\\
     &\Sigma X /\Sigma X_{2am} \ar[ur]_-{\bar F_1} ,
     }
$$
where $\Sigma p_{\zeta_m}$ is the projection and $\bar F_1$ is the extension obtained from the cofibration
$$
\xymatrix@C=13mm @R=5mm{
\Sigma X_{2am} \ar@{^{(}->}[r]^-{\Sigma \zeta_m} &\Sigma X   \ar[r]^-{\Sigma p_{\zeta_m}} &\Sigma X /\Sigma X_{2am}.
}
$$
The Remark \ref{Res} suggests that the extension
$$
\bar F_1 : \Sigma X /\Sigma X_{2am} \rightarrow \Sigma X
$$
will be possible.
We now consider the following commutative diagram:
$$
\xymatrix@C=15mm @R=8mm{
S^{2am+1} \ar[r]^-{x_m} \ar[d]^{\Delta} &\Sigma X  \ar[d]^{\Delta} \ar[dr]^{(1 \times \Sigma p_{\zeta_m}) \circ \Delta}\\
S^{2am+1} \times S^{2am+1} \ar[r]^-{x_m \times x_m}  &\Sigma X \times \Sigma X \ar[r]^-{1 \times \Sigma p_{\zeta_m}}  &\Sigma X \times \Sigma X /\Sigma X_{2am},
}
$$
where $\Delta$ is the diagonal map.
Since $\nu : S^{2am+1} \rightarrow S^{2am+1} \vee S^{2am+1}$ is the unique standard comultiplication and $\nu : \Sigma X \rightarrow \Sigma X \vee \Sigma X$ is the standard comultiplication, the following triangles
$$
\xymatrix@C=3.0mm @R=5mm{
S^{2am+1} \ar[rr]^-{\Delta}  \ar[dr]_-{\nu}  && S^{2am+1} \times S^{2am+1} & \mathrm{and} &\Sigma X \ar[rr]^-{\Delta}  \ar[dr]_-{\nu}  && \Sigma X \times \Sigma X &\\
&S^{2am+1} \vee S^{2am+1} \ar@{^{(}->}[ur]_-{\mathrm{inclusion}}& && &\Sigma X \vee \Sigma X \ar@{^{(}->}[ur]_-{\mathrm{inclusion}} }
$$
are homotopy commutative. Since $\Sigma X /\Sigma X_{2am}$ is $2am$-connected, from the fibration
$$
\xymatrix@C=10mm@R=8mm{
\Sigma X  \flat \Sigma X \ar[r]^-{} &\Sigma X  \vee \Sigma X \ar[r]^-{} &\Sigma X  \times \Sigma X,
}
$$
where $\flat$ is the flat product, we have an isomorphism of homotopy groups
$$
\xymatrix@C=10mm@R=8mm{
\pi_{2am+1} (\Sigma X  \vee \Sigma X /\Sigma X _{2am}) \ar[r]^-{\cong} &\pi_{2am+1} (\Sigma X  \times \Sigma X /\Sigma X_{2am})
}
$$
sending the homotopy class of
$$
(1 \vee \Sigma p_{\zeta_m})\circ (x_{m} \vee x_{m} ) \circ \nu \simeq (1 \vee \Sigma p_{\zeta_m})\circ \nu \circ x_{m}
$$
to the homotopy class of
$$
(1 \times \Sigma p_{\zeta_m}) \circ (x_{m} \times x_{m} ) \circ \Delta \simeq (1 \times \Sigma p_{\zeta_m})\circ \Delta \circ x_m .
$$
We finally have the following commutative diagram up to homotopy
$$
\xymatrix@C=28mm @R=10mm{
S^{2am+1} \ar[r]^-{x_m} \ar[d]^{\nu} &\Sigma X \ar[d]^-{(1 \vee \Sigma p_{\zeta_m})\circ \nu } \ar[r]^-{1 + F_1} &\Sigma X\\
S^{2am+1} \vee S^{2am+1} \ar[r]^-{(1 \vee \Sigma p_{\zeta_m})\circ (x_m \vee x_m)}  &\Sigma X \vee \Sigma X/ \Sigma X_{2am} \ar[ur]_-{(1, \bar F_1)}
}
$$
as required.
\end{proof}

\begin{rmk} Similarly, if $w_n$ and $w_s$ are homotopy classes of $\pi_{*}(\Sigma X)/{\rm torsion}$, then it can also be shown that
$$
{(1 + [\psi_{s_{l_2}}, [\psi_{s_{l_2 -1}},\ldots,[\psi_{s_{1}}, \psi_{s_2}]_c \ldots]_c]_c )}_{\sharp} (w_n ) \,
=\, w_n + {[\psi_{s_{l_2}}, [\psi_{s_{l_2 -1}},\ldots,[\psi_{s_{1}}, \psi_{s_2}]_c \ldots]_c]_c}_{\sharp} (w_n)
$$
and
$$
{(1 + [f_{s_{l}}, [f_{s_{l -1}},\ldots,[f_{s_{1}}, f_{s_2}]_c \ldots]_c]_c )}_{\sharp} (w_s ) \,
=\, w_s + {[f_{s_{l}}, [f_{s_{l -1}},\ldots,[f_{s_{1}}, f_{s_2}]_c \ldots]_c]_c}_{\sharp} (w_s)
$$
by using the same method in Proposition \ref{Re}.
\end{rmk}

Just like the (iterated) commutator on the suspension structure, we also consider the commutator on the loop structure
$$
[\hat f_{s_1}, \hat f_{s_2}]_c : X \longrightarrow \Omega \Sigma X
$$
defined by
$$
[\hat f_{s_1}, \hat f_{s_2}]_c (x) = \hat f_{s_1}(x) \cdot \hat f_{s_2}(x) \cdot \hat f_{s_1}^{-1}(x) \cdot \hat f_{s_2}^{-1}(x)
$$
where
\begin{enumerate}
\item [{\rm (1)}] $\hat f_{s_1} = \hat \varphi_{s_1}$ or $\hat \psi_{s_1}$ and  $\hat f_{s_2} = \hat \varphi_{s_2}$ or $\hat \psi_{s_2}$ for $s_1, s_2 = 1,2,3,\ldots$;
\item [{\rm (2)}] the operation `$\cdot$' is the loop multiplication, and the inverse is the loop inverse $\nu : \Omega \Sigma X \rightarrow \Omega \Sigma X$ defined by $\nu(g) = g^{-1}$ with $g^{-1} (s) = g(1-s)$.
\end{enumerate}
By induction, we define the (pure and hybrid) iterated commutator
$$
[\hat f_{s_l}, [\hat f_{s_{l-1}},\ldots,[\hat f_{s_{1}},\hat f_{s_2}]_c \ldots]_c]_c : X \longrightarrow \Omega \Sigma X
$$
of maps $\hat f_{s_{t}} : X \rightarrow \Omega \Sigma X, t =1,2,\ldots,l$ on the loop structure again, where $f_{s_t} = \varphi_{s_t}$ or $\psi_{s_t}$ for $s_t = 1,2,3,\ldots$.

\begin{prop} \label{Prop1}
Let $\hat \varphi_1 : X \rightarrow \Omega \Sigma X$ be the canonical inclusion. Then
$$
\Omega [f_{s_l}, [f_{s_{l-1}},\ldots,[f_{s_{1}}, f_{s_2}]_c \ldots]_c]_c  \circ \hat \varphi_1
= [\hat f_{s_l}, [\hat f_{s_{l-1}},\ldots,[\hat f_{s_{1}}, \hat f_{s_2}]_c \ldots]_c]_c ,
$$
where $f_{s_t} = \varphi_{s_t}$ or $\psi_{s_t}$ for $t =1,2,\ldots,l$ and $s_t = 1,2,3,\ldots$.
\end{prop}

\begin{proof} We note that the canonical inclusion
$$
\hat \varphi_1 : X \rightarrow \Omega \Sigma X
$$
sends
$$
x \mapsto \hat \varphi_1 (x) : I \rightarrow \Sigma X,
$$
with
$$
\hat \varphi_1(x)(t) = [t,x]
$$
in $\Sigma X$, where $[t,x] \in \Sigma X$. Similarly, the maps
$$
\Omega [f_{s_l}, [f_{s_{l-1}},\ldots,[f_{s_{1}}, f_{s_2}]_c \ldots]_c]_c  \circ \hat \varphi_1 : X \longrightarrow \Omega \Sigma X,
$$
sends
$$
x \mapsto \Omega [f_{s_l}, [f_{s_{l-1}},\ldots,[f_{s_{1}}, f_{s_2}]_c \ldots]_c]_c  \circ \hat \varphi_1 (x) : I \rightarrow \Sigma X,
$$
and
\begin{equation}\tag{2.6}
((\Omega [f_{s_l}, [f_{s_{l-1}},\ldots,[f_{s_{1}}, f_{s_2}]_c \ldots]_c]_c  \circ \hat \varphi_1)(x))(t)
= [f_{s_l}, [f_{s_{l-1}},\ldots,[f_{s_{1}}, f_{s_2}]_c \ldots]_c]_c ([ t, x]),
\end{equation}
for all $t \in I$ and $x \in X$.

The adjointness $(\Sigma, \Omega)$ shows that there is an isomorphism
$$
\text{ad} = \widehat ~~: [\Sigma X, \Sigma X] \rightarrow [X, \Omega \Sigma X]
$$
of groups (not necessarily abelian), and we obtain
\begin{equation}\tag{2.7}
\text{ad} [f_{s_l}, [f_{s_{l-1}},\ldots,[f_{s_{1}}, f_{s_2}]_c \ldots]_c]_c
= [\hat f_{s_l}, [\hat f_{s_{l-1}},\ldots,[\hat f_{s_{1}}, \hat f_{s_2}]_c \ldots]_c]_c  : X \rightarrow \Omega \Sigma X
\end{equation}
as homotopy classes. Therefore, we have
\begin{equation}\tag{2.8}
(\text{ad} [f_{s_l}, [f_{s_{l-1}},\ldots,[f_{s_{1}}, f_{s_2}]_c \ldots]_c]_c (x))(t) =
[f_{s_l}, [f_{s_{l-1}},\ldots,[f_{s_{1}},f_{s_2}]_c \ldots]_c]_c ([t,x])
\end{equation}
for all $t \in I$ and $x \in X$.
By (2.6), (2.7) and (2.8), we complete the proof.
\end{proof}

\medskip

\section{Iterated Samelson (or Whitehead) products and homotopy generators}\label{section3}

Let ${\rm Aut} (Y^{(n)})$ denote the discrete group of homotopy classes of all
homotopy self-equivalences of $Y^{(n)}$
and let ${\rm Aut}( \pi_{\leq n} (Y))$ be the group of automorphisms of a graded group $\pi_{\leq n} (Y)$ that preserve the pairs of the Whitehead products. In 1992, McGibbon and M{\o}ller \cite[Theorem 1]{MM} proved the following theorem as the Eckmann-Hilton dual of \cite[Theorem 3]{MM1}:

\begin{thm} \label{MM} Let $Y$ be a simply connected space with finite type
over some subring of the ring of rational numbers. Assume that $Y$ has the rational
homotopy type of a bouquet of spheres. Then the following three conditions are equivalent:
\begin{enumerate}
\setlength{\itemsep}{-3mm}
\item[{\rm (a)}] $SNT(Y) = *$;
\item[{\rm (b)}] the map $\xymatrix@C=3.5pc{
{\rm Aut} (Y) \ar[r]^-{f \mapsto f^{(n)}} &{\rm Aut}(Y^{(n)})}$ has a finite cokernel for each $n$; and
\item[{\rm (c)}] the map $\xymatrix@C=3.5pc{
{\rm Aut} (Y) \ar[r]^-{f \mapsto f_\sharp} &{\rm Aut}( \pi_{\leq n} (Y))}$ has a finite cokernel for each $n$,
\end{enumerate}
where $*$ denotes the set consisting of a single homotopy type $\{[Y]\}$.
\end{thm}

We note that the Leray-Serre spectral sequence of a path space fibration of the Eilenberg-MacLane spaces says that
the rational cohomology $H^* (K(\mathbb Z, 2a) ;\mathbb Q)$ is isomorphic to the polynomial algebra $\mathbb Q[\alpha]$ over $\mathbb Q$ generated by $\alpha$ of dimension $2a$; that is, $\alpha$
is a rational generator of $H^{2a}(K(\mathbb Z, 2a);\mathbb Q)$ satisfying $<\alpha^s ,\bar \alpha_t > = \delta_{st}$, where $\bar \alpha_t$ is a rational
homology generator of dimension $2at$. Similarly,
there is an isomorphism of algebras
$$
H^* (K(\mathbb Z, 2aj+1) ;\mathbb Q) \cong \Lambda[\beta_j],
$$
where $\Lambda[\beta_j]$ is the exterior algebra on the cohomology class $\beta_j$ of odd dimension $2aj+1$ for each $j =1,2,3,\ldots$.

Generally, if $X$ is a connected Hopf space of finite type, then it has $k$-invariants of finite order, and its rational cohomology is isomorphic to the tensor product of the polynomial algebra of even degree generators and the exterior algebra of odd degree generators as algebras.
We note that the multiplication map for the loop structure
$\mu : \Omega \Sigma X \times \Omega \Sigma X \longrightarrow \Omega \Sigma X$
provides the homology and cohomology of $\Omega \Sigma X$ with an algebra structure that is natural with respect to Hopf maps.
Specifically, the loop structure of $\Omega \Sigma X$ makes the graded rational homology $H_* (\Omega \Sigma X; \mathbb Q)$ into the graded algebra which is called the {\it Pontryagin algebra} of $\Omega \Sigma X$.

In reduced homology of $X: = K( {\mathbb Z}, 2a) \vee \bigvee_{j=1}^{\infty} K( {\mathbb Z}, 2aj+1)$ with rational coefficients, we obtain
$$
\tilde H_* ( X ; {\mathbb Q}) \cong {\mathbb Q} \{ u_1, u_2, u_3, \ldots, u_i, \ldots ; v_1, v_2, v_3, \ldots, v_j , \ldots \}
$$
as a graded $\mathbb Q$-module, where $u_i \in H_{2ai}( X ; {\mathbb Q})$ and $v_j \in H_{2aj+1}( X ; {\mathbb Q})$ are the
standard generators of rational homology groups for $i,j =1,2,3, \ldots$.
It is well known in \cite{BS} that the Pontryagin algebra
$H_* (\Omega \Sigma X ;\mathbb Q )$ is isomorphic to the tensor algebra
$TH_* (X ; \mathbb Q )$; that is, the rational homology of $\Omega
\Sigma X $ is the tensor algebra
$$
TH_* (X ; \mathbb Q ) \cong T [u_1, u_2, u_3, \ldots, u_i, \ldots ; v_1, v_2, v_3, \ldots, v_j , \ldots]
$$
generated by  $\{ u_1, u_2, u_3, \ldots, u_i, \ldots ; v_1, v_2, v_3, \ldots, v_j , \ldots \}$.

For the moment, we will make use of the following notations:
\begin{itemize}
\item $\iota_{s_1, s_2, \ldots, s_l} : X_{2a(s_1+s_2+\cdots+s_l)} \hookrightarrow X$ is the inclusion.
\item $\xi_{s_1, s_2, \ldots, s_l}^{l-t} : X_{2a(s_1+s_2+\cdots+s_l)+t} \hookrightarrow X$ is the inclusion for $1 \leq t \leq l-1$.
\item $\zeta_{s_1, s_2, \ldots, s_l} : X_{2a(s_1+s_2+\cdots+s_l)+l} \hookrightarrow X$ is the inclusion.
\item $\pi_{\iota_{s_1, s_2, \ldots, s_l}} : X_{2a(s_1+s_2+\cdots+s_l)} \rightarrow S^{2a(s_1+s_2+\cdots+s_l)}$ is the projection to the top cell of $X_{2a(s_1+s_2+\cdots+s_l)}$.
\item $\pi_{\xi_{s_1, s_2, \ldots, s_l}^{l-t}} : X_{2a(s_1+s_2+\cdots+s_l)+t} \rightarrow S^{2a(s_1+s_2+\cdots+s_l)+t}$ is the projection to the top cell of $X_{2a(s_1+s_2+\cdots+s_l)+t}$ for $1 \leq t \leq l-1$.
\item $\pi_{\zeta_{s_1, s_2, \ldots, s_l}} : X_{2a(s_1+s_2+\cdots+s_l)+l} \rightarrow S^{2a(s_1+s_2+\cdots+s_l)+l}$ is the projection to the top cell of $X_{2a(s_1+s_2+\cdots+s_l)+l}$.
\item $\bar \Delta : X \rightarrow X \wedge X$ be the reduced diagonal map; that is, the composite of the diagonal $\Delta : X \rightarrow X \times X$ with the natural projection $X \times X \rightarrow X \wedge X$ onto the smash product.
\end{itemize}

We now have the following:

\begin{lem} \label{Le0} If $<~,~>$ is the Samelson product, then
\begin{itemize}
\item [{\rm (a)}] $[\hat \varphi_{s_{i}}, \hat \varphi_{s_j}]_c \circ \iota_{s_i, s_j} ~\simeq~ <\hat x_{s_{i}}, \hat x_{s_{j}}> \circ \pi_{\iota_{s_i, s_j}} : X_{2a(s_i+s_j)} \rightarrow \Omega \Sigma X$;

\item [{\rm (b)}] $[\hat \varphi_{s_{i}}, \hat \psi_{s_j}]_c \circ \xi_{s_i, s_j}^1 ~\simeq~ <\hat x_{s_{i}}, \hat y_{s_{j}}> \circ \pi_{\xi_{s_i, s_j}^1} : X_{2a(s_i+s_j)+1} \rightarrow \Omega \Sigma X$; and

\item [{\rm (c)}] $[\hat \psi_{s_{i}}, \hat \psi_{s_j}]_c \circ \zeta_{s_i, s_j} ~\simeq~ <\hat y_{s_{i}}, \hat y_{s_{j}}> \circ \pi_{\zeta_{s_i, s_j}} : X_{2a(s_i+s_j)+2} \rightarrow \Omega \Sigma X$
\end{itemize}
for $i,j = 1,2,\ldots,l$ with $l \geq 2$ and $s_i, s_j =1,2,3,\ldots$.
\end{lem}

\begin{proof} We prove the second part; the proofs of the first and last parts are similar with the second one.
Let $p_{\iota_{s_i}} : X \rightarrow X/ X_{2as_i -1}$ and $p_{\zeta_{s_j}} : X \rightarrow X/ X_{2as_j}$ be the natural projections. Then the composite
$$
\xymatrix@C=12mm@R=1mm{
X_{2a(s_i+s_j)}  \ar@{^{(}->}[r]^-{\zeta_{s_i + s_j}} &X  \ar[r]^-{\bar \Delta}  &X \wedge X \ar[r]^-{p_{\iota_{s_i}} \wedge p_{\zeta_{s_j}}} &X/X_{2as_i -1} \wedge X/X_{2as_j}
}
$$
is inessential by the cellular approximation theorem and by the cell structure of $X \wedge X$, i.e.,
$$
(p_{\iota_{s_i}} \wedge p_{\zeta_{s_j}}) \circ \bar \Delta \circ \zeta_{s_i + s_j} \simeq *.
$$
Therefore, there is a map
$$
\Phi : X/X_{2a(s_i+s_j)} \longrightarrow X/X_{2as_i -1} \wedge X/X_{2as_j}
$$
satisfying
$$
\Phi \circ p_{\zeta_{s_i + s_j}} \simeq p_{\iota_{s_i}} \wedge p_{\zeta_{s_j}} \circ \bar \Delta,
$$
where $p_{\zeta_{s_i + s_j}} : X \rightarrow X/X_{2a(s_i+s_j)}$ is the projection; that is, the map $\Phi$ makes the following diagram
$$
\xymatrix@R=2.7pc @C=3.2pc{
X_{2a(s_i+s_j)}  \ar@{^{(}->}[d]_-{\zeta_{s_i + s_j}} \ar[drr]^-{*} \\
X  \ar[r]^-{\bar \Delta}  \ar[d]_-{p_{\zeta_{s_i + s_j}}} &X \wedge X \ar[r]^-{p_{\iota_{s_i}} \wedge p_{\zeta_{s_j}}} &X/X_{2as_i -1} \wedge X/X_{2as_j}\\
X/X_{2a(s_i+s_j)} \ar[urr]_-{\Phi}}
$$
commute up to homotopy.
This is due to the exact sequence
$$
\xymatrix@C=12mm@R=1mm{
[X/X_{2a(s_i+s_j)} , Y] \ar[r]^-{p_{\zeta_{s_i + s_j}}^\sharp} &[X, Y] \ar[r]^-{\zeta_{s_i + s_j}^\sharp} &[X_{2a(s_i+s_j)}, Y]}
$$
induced by a cofibration
$$
\xymatrix@C=12mm@R=1mm{
X_{2a(s_i+s_j)}~ \ar@{^{(}->}[r]^-{\zeta_{s_i + s_j}} &X \ar[r]^-{p_{\zeta_{s_i + s_j}}} & X/X_{2a(s_i+s_j)},}
$$
where $Y = X/X_{2as_i -1} \wedge X/X_{2as_j}$.

Let
$$
C : \Omega \Sigma X \wedge \Omega \Sigma X \rightarrow \Omega \Sigma X
$$
be the commutator map with respect to the loop multiplication. Then
the following diagram
$$
\xymatrix@R=3.2pc @C=3.2pc{
X_{2a(s_i+s_j)}  \ar@{^{(}->}[r]^-{\zeta_{s_i + s_j}} \ar@{^{(}->}[d]^-{} & X  \ar[r]^-{\bar \Delta}  \ar[d]^-{p_{\zeta_{s_i + s_j}}} &X \wedge X
\ar[dr]^-{\hat \varphi_{s_{i}} \wedge \hat \psi_{s_j}} \ar[d]_-{p_{\iota_{s_i}} \wedge p_{\zeta_{s_j}}} \\
X_{2a(s_i+s_j)+1} \ar@{^{(}->}[ur]^-{\xi_{s_i, s_j}^1} \ar[r]^-{}  \ar[dr]_-{\pi_{\xi_{s_i, s_j}^1}} & X/X_{2a(s_i+s_j)}  \ar[r]^-{\Phi}
&X/X_{2as_i -1} \wedge X/X_{2as_j}
\ar[r]^-{\bar \varphi_{s_{i}} \wedge \bar \psi_{s_j}}  &\Omega \Sigma X \wedge \Omega \Sigma X \ar[d]^-{C} \\
&S^{2a(s_i+s_j)+1} \ar[u] \ar[r]^-{\approx}  &S^{2as_i} \wedge S^{2as_j+1} \ar[u] \ar[ur]^-{\hat x_{s_i} \wedge \hat y_{s_j}} \ar[r]^-{<\hat x_{s_i}, \hat y_{s_j} >} &\Omega \Sigma X }
$$
is homotopy commutative (see also \cite{Mo} in the special case of the infinite complex projective space). Here, the commutativity of the three triangles on the right-hand side of the diagram above is guaranteed by
\begin{itemize}
\item [{\rm (1)}] the constructions of $\hat \varphi_{s_{i}}$, $\hat \psi_{s_j}$, $\bar \varphi_{s_{i}}$ and $\bar \psi_{s_j}$ in (2.3) and (2.4);
\item [{\rm (2)}] the definitions of $\hat x_{s_{i}}$ and $\hat y_{s_{j}}$ in Definition \ref{def2.2}; and
\item [{\rm (3)}] the Samelson products of $\hat x_{s_{i}}$ and $\hat y_{s_{j}}$
\end{itemize}
for $i,j = 1,2,\ldots,l$ with $l \geq 2$ and $s_i, s_j =1,2,3,\ldots$. The homotopy commutative diagram above shows that
$$
[\hat \varphi_{s_i}, \hat \psi_{s_j} ]_c \circ \xi_{s_i, s_j}^1 =
<\hat x_{s_i}, \hat y_{s_j} > \circ ~\pi_{\xi_{s_i, s_j}^1}
$$
as required.
\end{proof}

We now consider the general case as follows:

\begin{thm} \label{Thm}
Let $<,<,\ldots, <~,~>\ldots>>$ denote the iterated Samelson product.

{\rm(a)} If the iterated commutators $[\hat f_{s_l}, [\hat f_{s_{l-1}},\ldots,[\hat f_{s_{1}}, \hat f_{s_2}]_c \ldots]_c]_c$ are pure, where $l = l_1$ or $l_2$ with $l_1, l_2 \geq 2$, then
$$
[\hat \varphi_{s_{l_1}}, [\hat \varphi_{s_{l_1 -1}},\ldots,[\hat \varphi_{s_{1}}, \hat \varphi_{s_2}]_c \ldots]_c]_c  \circ \iota_{s_1,\ldots,s_{l_1}} \simeq
<\hat x_{s_{l_1}}, <\hat x_{s_{l_1} -1},\ldots,<\hat x_{s_{1}}, \hat x_{s_2}> \ldots>>
\circ \pi_{\iota_{s_1,\ldots,s_{l_1}}}
$$
and
$$
[\hat \psi_{s_{l_2}}, [\hat \psi_{s_{l_2 -1}},\ldots,[\hat \psi_{s_{1}}, \hat \psi_{s_2}]_c \ldots]_c]_c  \circ \zeta_{s_1,\ldots,s_{l_2}} \simeq
<\hat y_{s_{l_2}}, <\hat y_{s_{l_2} -1},\ldots,<\hat y_{s_{1}}, \hat y_{s_2}> \ldots>>
\circ \pi_{\zeta_{s_1,\ldots,s_{l_2}}}.
$$

{\rm(b)} If the iterated commutator $[\hat f_{s_l}, [\hat f_{s_{l-1}},\ldots,[\hat f_{s_{1}}, \hat f_{s_2}]_c \ldots]_c]_c$ is hybrid which contains both the $t$-times of $\varphi_{s_{i}}$ and the $(l-t)$-times of $\psi_{s_{j}}$ for $1 \leq t \leq l-1$, then
$$
[\hat f_{s_l}, [\hat f_{s_{l-1}},\ldots,[\hat f_{s_{1}}, \hat f_{s_2}]_c \ldots]_c]_c  \circ \xi_{s_1,\ldots,s_l}^{l-t} ~\simeq~
<\hat z_{s_l}, <\hat z_{s_{l-1}},\ldots,<\hat z_{s_{1}}, \hat z_{s_2}> \ldots>>
\circ \pi_{\xi_{s_1,\ldots,s_l}^{l-t}}.
$$
Here,
\begin{itemize}
\item [{\rm (1)}] $\hat f_{s_{i_1}} = \hat \varphi_{s_{i_1}}, \hat f_{s_{i_2}} = \hat \varphi_{s_{i_2}}, \ldots, \hat f_{s_{i_t}} = \hat \varphi_{s_{i_t}}$;
\item [{\rm (2)}] $\hat f_{s_{j_1}} = \hat \psi_{s_{j_1}}, \hat f_{s_{j_2}} = \hat \psi_{s_{j_2}}, \ldots, \hat f_{s_{j_{l-t}}} = \hat \psi_{s_{j_{l-t}}}$;
\item [{\rm (3)}] $\hat z_{s_{i_1}} = \hat x_{s_{i_1}}, \hat z_{s_{i_2}} = \hat x_{s_{i_2}}, \ldots, \hat z_{s_{i_t}} = \hat x_{s_{i_t}}$; and
\item [{\rm (4)}] $\hat z_{s_{j_1}} = \hat y_{s_{j_1}}, \hat z_{s_{j_2}} = \hat y_{s_{j_2}}, \ldots, \hat z_{s_{j_{l-t}}} = \hat y_{s_{j_{l-t}}}$,
\end{itemize}
where $(s_{i_1}, s_{i_2}, \ldots, s_{i_t}, s_{j_1}, s_{j_2}, \ldots, s_{j_{l-t}})$, $(s_{i_1}, s_{i_2}, \ldots, s_{i_t})$ and $(s_{j_1}, s_{j_2}, \ldots, s_{j_{l-t}})$ are permutations of the symmetric groups $S_{\{s_1, s_2, \ldots, s_l\}}$ of degree $l$, $S_{\{s_1, s_2, \ldots, s_l\}- \{s_{j_1},s_{j_2},\ldots,s_{j_{l-t}} \}}$ of degree $t$, and $S_{\{s_{j_1},s_{j_2},\ldots,s_{j_{l-t}} \}}$ of degree $l-t$, respectively.
\end{thm}

\begin{proof}
We argue by induction on $l_1$ ($l_2$ or $l$).
The proof in case of the $2$-fold commutators and the $2$-fold Samelson products is followed by Lemma \ref{Le0}.
We suppose the results are true for the (pure and hybrid) iterated commutators and the iterated Samelson products of length $l_1 -1$ (resp. $l_2 -1$ or $l-1$). Then there exists a map
$$
\Psi : X/X_{2a(s_1+s_2 + \cdots +s_{l_1})-1} \longrightarrow X/X_{2as_{l_1} -1} \wedge X/X_{2a(s_1+s_2 + \cdots + s_{l_{1}-1}) -1}
$$
such that the following diagram
$$
\xymatrix@C=8mm @R=8mm{
X_{2a(s_1+s_2 + \cdots +s_{l_1})-1}  \ar@{^{(}->}[d]^-{\iota_{(s_1+s_2 + \cdots +s_{l_1})}} \\
X  \ar[r]^-{\bar \Delta}  \ar[d]^-{p_{\iota_{(s_1+s_2 + \cdots +s_{l_1})}}} &X \wedge X
\ar[d]^-{p_{\iota_{s_{l_1}}} \wedge p_{\iota_{(s_1+s_2 + \cdots + s_{l_1 -1})}}} \\
X/X_{2a(s_1+s_2 + \cdots +s_{l_1})-1}  \ar[r]^-{\Psi} &X/X_{2as_{l_1} -1} \wedge X/X_{2a(s_1+s_2 + \cdots + s_{l_{1}-1}) -1} \\
}
$$
commutes up to homotopy (similarly for the other cases).
That is, the proof in case of the $l_1$-fold (resp. $l_2$-fold or $l$-fold) iterated commutators and the corresponding $l_1$-fold (resp. $l_2$-fold or $l$-fold) iterated Samelson products can be proven inductively on a case-by-case basis by substituting $\hat \varphi_{s_{l_1}}$ (resp. $\hat \psi_{s_{l_2}}$ or $\hat f_{s_l}$) and
$[\hat \varphi_{s_{l_1 -1}},\ldots,[\hat \varphi_{s_{1}}, \hat \varphi_{s_2}]_c \ldots]_c$
(resp. $[\hat \psi_{s_{l_2 -1}},\ldots,[\hat \psi_{s_{1}}, \hat \psi_{s_2}]_c \ldots]_c$ or $[\hat f_{s_{l -1}},\ldots,[\hat f_{s_{1}}, \hat f_{s_2}]_c \ldots]_c$)
for $\hat \varphi_{s_{1}}$ (resp. $\hat \psi_{s_{1}}$ or $\hat f_{s_{1}}$) and $\hat \varphi_{s_{2}}$ (resp. $\hat \psi_{s_{2}}$ or $\hat f_{s_{2}}$), respectively, as just proved in Lemma \ref{Le0}.
\end{proof}

In the rest of this paper, we will abbreviate notations by writing the following in order to avoid repeating the same letters (except for the hybrid iterated commutators and their corresponding Samelson products and Whitehead products).
\begin{itemize}
\item $[\hat \varphi_{s_{l_1}}, [\hat \varphi_{s_{l_1 -1}},\ldots,[\hat \varphi_{s_{1}}, \hat \varphi_{s_2}]_c \ldots]_c]_c
     = \hat \varphi_{[s_{l_1}, [s_{l_1 -1},\ldots,[s_{1}, s_2 ]_c \ldots]_c]_c}$;
\item $[\hat \psi_{s_{l_2}}, [\hat \psi_{s_{l_2 -1}},\ldots,[\hat \psi_{s_{1}}, \hat \psi_{s_2}]_c \ldots]_c]_c
     = \hat \psi_{[s_{l_2}, [s_{l_2 -1},\ldots,[s_{1}, {s_2}]_c \ldots]_c]_c}$;
\item $[\varphi_{s_{l_1}}, [\varphi_{s_{l_1 -1}},\ldots,[\varphi_{s_{1}}, \varphi_{s_2}]_c \ldots]_c]_c
     = \varphi_{[s_{l_1}, [s_{l_1 -1},\ldots,[s_{1}, s_2 ]_c \ldots]_c]_c}$;
\item $[\psi_{s_{l_2}}, [\psi_{s_{l_2 -1}},\ldots,[\psi_{s_{1}}, \psi_{s_2}]_c \ldots]_c]_c
     = \psi_{[s_{l_2}, [s_{l_2 -1},\ldots,[s_{1}, {s_2}]_c \ldots]_c]_c}$;
\item $<\hat x_{s_{l_1}}, <\hat x_{s_{l_1 -1}},\ldots,<\hat x_{s_{1}}, \hat x_{s_2}> \ldots > >
     = \hat x_{<s_{l_1}, <s_{l_1 -1},\ldots,<s_{1}, s_2 > \ldots>>}$;
\item $<\hat y_{s_{l_2}}, <\hat y_{s_{l_2 -1}},\ldots, <\hat y_{s_{1}}, \hat y_{s_2}> \ldots>>
     = \hat y_{<s_{l_2}, <s_{l_2 -1},\ldots,<s_{1}, {s_2}> \ldots>>}$;
\item $[x_{s_{l_1}}, [x_{s_{l_1 -1}},\ldots,[x_{s_{1}}, x_{s_2}] \ldots ]]
     = x_{[s_{l_1}, [s_{l_1 -1},\ldots,[s_{1}, s_2] \ldots]]}$; and
\item $[y_{s_{l_2}}, [y_{s_{l_2 -1}},\ldots, [y_{s_{1}}, y_{s_2}] \ldots]]
     = y_{[s_{l_2}, [s_{l_2 -1},\ldots,[s_1, s_2] \ldots]]}$.
\end{itemize}

\begin{cor} \label{Cor} Let $h : \pi_* (\Omega \Sigma X) \rightarrow H_* (\Omega \Sigma X ; \mathbb Q)$ be the Hurewicz homomorphism.

{\rm(a)}
If the iterated commutator $[\hat f_{s_l}, [\hat f_{s_{l-1}},\ldots,[\hat f_{s_{1}}, \hat f_{s_2}]_c \ldots]_c]_c$ is pure, where $l = l_1$ or $l_2$ with $l_1, l_2 \geq 2$, then
$$
\hat \varphi_{{[s_{l_1}, [s_{l_1 -1},\ldots,[s_1, s_2 ]_c \ldots]_c ]_c}_*}   (u_m)
= h (\hat x_{<s_{l_1}, <s_{l_1 -1},\ldots,<s_1, s_2 > \ldots>>} );
$$
and
$$
\hat \psi_{{[s_{l_2}, [s_{l_2 -1},\ldots,[s_{1}, {s_2}]_c \ldots]_c ]_c}_*} (w_n)
= h( \hat y_{<s_{l_2}, <s_{l_2 -1},\ldots,<s_{1}, {s_2}> \ldots>>} )
$$
for the rational homology generators $u_m$ and $w_n$ of dimensions $2a(s_1 + s_2 +\cdots + s_{l_1})$ and $2a(s_1 + s_2 +\cdots + s_{l_2}) +l_2$, respectively. Here
\begin{itemize}
\item [{\rm (1)}] $m = s_1 + s_2 +\cdots + s_{l_1}$; and
\item [{\rm (2)}] if $l_2$ is odd, namely, $l_2 = 2ab +1$, then $w_n = v_n$ and $n = s_1 + s_2 +\cdots + s_{l_2}+b$, or if $l_2$ is even, namely, $l_2 = 2ab$, then $w_n = u_n$ and $n = s_1 + s_2 +\cdots + s_{l_2}+b$.
\end{itemize}

{\rm(b)} Let $[\hat f_{s_l}, [\hat f_{s_{l-1}},\ldots,[\hat f_{s_{1}}, \hat f_{s_2}]_c \ldots]_c]_c$, $\hat f_{s_{i_1}}, \hat f_{s_{j_1}}, \hat z_{s_{i_1}}$ and $\hat z_{s_{j_1}}$ be as in Theorem \ref{Thm}. Then
$$
{[\hat f_{s_l}, [\hat f_{s_{l-1}},\ldots,[\hat f_{s_{1}}, \hat f_{s_2}]_c \ldots]_c]_c}_* (w_s )
= h(<\hat z_{s_l}, <\hat z_{s_{l-1}},\ldots,<\hat z_{s_{1}}, \hat z_{s_2}> \ldots>> )
$$
for a rational generator $w_s$ of dimension $2a(s_1 + s_2 +\cdots + s_l) + l-t$ in rational homology.
\end{cor}

\begin{proof}
By applying the results of Theorem \ref{Thm} to the homology with rational coefficients, we have the proof.
\end{proof}

We note that $h (<\hat x_{s_{1}}, \hat x_{s_2}> ) = [h(\hat x_{s_{1}}), h(\hat x_{s_2})]$ in homology with rational coefficients \cite[page 141]{CMN}, where $[h(\hat x_{s_{1}}), h(\hat x_{s_2})] = h(\hat x_{s_{1}}) h(\hat x_{s_2}) - (-1)^{|h(\hat x_{s_{1}})| |h(\hat x_{s_2})|} h(\hat x_{s_{2}}) h(\hat x_{s_1})$.

Under the notations in Theorem \ref{Thm} and Corollary \ref{Cor}, we now have the following theorem as one of the fundamental ingredients for the proof of the main theorem in Section 4.

\begin{thm} \label{thm4} For each iterated Samelson product such as
$$
\hat x_{<s_{l_1}, <s_{l_1 -1},\ldots,<s_{1}, s_2 > \ldots>>} ~: S^{2am} \rightarrow \Omega \Sigma X,
$$
$$
\hat y_{<s_{l_2}, <s_{l_2 -1},\ldots,<s_{1}, {s_2}> \ldots>>} ~: S^{2an+\delta} \rightarrow \Omega \Sigma X ~(\delta = 1,2)
$$
and
$$
<\hat z_{s_l}, <\hat z_{s_{l-1}},\ldots,<\hat z_{s_{1}}, \hat z_{s_2}> \ldots>> ~: S^{2a(s_1 + s_2 +\cdots + s_l)+l-t} \rightarrow \Omega \Sigma X
$$
in the graded homotopy group $\pi_* (\Omega \Sigma X)$, there exist iterated commutator maps
$$
\varphi_{[s_{l_1}, [s_{l_1 -1},\ldots,[s_{1}, s_2 ]_c \ldots]_c]_c} : \Sigma X \rightarrow \Sigma X,
$$
$$
\psi_{[s_{l_2}, [s_{l_2 -1},\ldots,[s_{1}, {s_2}]_c \ldots]_c]_c} : \Sigma X \rightarrow \Sigma X
$$
and
$$
[ f_{s_l}, [ f_{s_{l-1}},\ldots,[ f_{s_{1}},  f_{s_2}]_c \ldots]_c]_c : \Sigma X \rightarrow \Sigma X
$$
which correspond to the types of the iterated Samelson products such that
$$
{\Omega \varphi_{[s_{l_1}, [s_{l_1 -1},\ldots,[s_{1}, s_2 ]_c \ldots]_c]_c}}_\sharp (\hat x_m)
= \alpha \cdot \hat x_{<s_{l_1}, <s_{l_1 -1},\ldots,<s_{1}, s_2 > \ldots>>},
$$
$$
{\Omega \psi_{[s_{l_2}, [s_{l_2 -1},\ldots,[s_{1}, {s_2}]_c \ldots]_c]_c}}_\sharp (\hat w_n)
= \beta \cdot \hat y_{<s_{l_2}, <s_{l_2 -1},\ldots,<s_{1}, {s_2}> \ldots>>}
$$
and
$$
{\Omega [ f_{s_l}, [ f_{s_{l-1}},\ldots,[ f_{s_{1}},  f_{s_2}]_c \ldots]_c]_c}_\sharp (\hat w_s)= \gamma \cdot <\hat z_{s_l}, <\hat z_{s_{l-1}},\ldots,<\hat z_{s_{1}}, \hat z_{s_2}> \ldots>>.
$$
Here
\begin{itemize}
\item [{\rm (1)}] $\alpha, \beta$ and $\gamma$ are all nonzero;
\item [{\rm (2)}] $\hat x_m, \hat w_n$ and $\hat w_s$ are rationally nonzero homotopy classes of homotopy groups modulo torsion subgroups $\pi_{*}(\Omega \Sigma X)/{\rm torsion}$ in dimensions
$2a(s_1 + s_2 +\cdots + s_{l_1})$, $2a(s_1 + s_2 +\cdots + s_{l_2}) +l_2$ and $2a(s_1 + s_2 +\cdots + s_l) + l-t$, respectively, where
$m = s_1 + s_2 +\cdots + s_{l_1}$;
if $l_2$ is odd, namely, $l_2 = 2ab +1$, then $\hat w_n = \hat y_n$ and $n = s_1 + s_2 +\cdots + s_{l_2}+b$, or if $l_2$ is even, namely, $l_2 = 2ab$, then $\hat w_n = \hat x_n$ and $n = s_1 + s_2 +\cdots + s_{l_2}+b$;
$\hat w_s = \hat x_s$ or $\hat y_s$ depending on the parity of $2a(s_1 + s_2 +\cdots + s_l) +l -t$; and
\item [{\rm (3)}] $f_{s_i} : \Sigma X \rightarrow \Sigma X$ is the adjoint of $\hat f_{s_i} : X \rightarrow \Omega \Sigma X$ as one of the maps
  $\hat \varphi_{s_i}$ or $\hat \psi_{s_i}$; that is, $f_{s_i} = \varphi_{s_i}$ or $\psi_{s_i}$  for $i = 1,2,\ldots, l$.
\end{itemize}
\end{thm}

\begin{proof} To prove the theorem, we need the two lemmas as follows:

\begin{lem} \label{lem10} {\rm (Cartan and Serre)} Let $\mathbb{F}$ be an arbitrary field of characteristic zero. If $Y$ is a simply connected space of finite type, then the Hurewicz homomorphism induces a linear isomorphism
$$
\xymatrix@C=2.5pc{
\bar h : \pi_* (\Omega Y) \otimes \mathbb{F} \ar[r]^-{\cong} &PH_* (\Omega Y; \mathbb{F}),
}
$$
where $PH_* (\Omega Y; \mathbb{F})$ is the primitive subspace of $H_* (\Omega Y; \mathbb{F})$.
\end{lem}

\begin{proof} See \cite[page 231]{FHT} for details.
\end{proof}

\begin{lem}\label{omega} Let $C$ be a co-H-space and let $D$ be a CW-complex. Then the map
$$
\xymatrix@C=2.5pc{
\Omega : [C, D] \ar[r]^-{} & [\Omega C, \Omega D]_H,
}
$$
defined by
$\Omega ([f]) = [\Omega f]$
is a bijection as sets, where $[\Omega C, \Omega D]_H$ is the set of homotopy classes of Hopf maps.
\end{lem}

\begin{proof}
See \cite[page 75]{S}) for more details.
\end{proof}

Let
$r : \Omega \Sigma X \rightarrow \Omega \Sigma X_{\mathbb{Q}}$
be the rationalization map. Then we have the following commutative diagram:
$$
\xymatrix@C=15mm @R=10mm{
\ar@/^1.5pc/[rr]^-{h} \pi_* (\Omega \Sigma X ) \ar[r]^-{r_* } \ar[d]_-{\Omega F_\sharp}  &\pi_* (\Omega \Sigma X_{\mathbb{Q}} )  \ar[d]^-{\Omega F_{\mathbb{Q}_\sharp}} \ar[r]^-{\bar h}_-{\cong} &PH_* (\Omega \Sigma X ; \mathbb Q ) \ar[d]^-{\Omega F_*} \\
\ar@/_1.5pc/[rr]_-{h} \pi_* (\Omega \Sigma X ) \ar[r]^-{r_*}  &\pi_* (\Omega \Sigma X_{\mathbb{Q}} )  \ar[r]^-{\bar h}_-{\cong} &PH_* (\Omega \Sigma X ; \mathbb Q ),
}
$$
where $h = \bar h \circ r_*$, $F_{\mathbb{Q}} : \Sigma X_{\mathbb{Q}} \rightarrow \Sigma X_{\mathbb{Q}}$, and
$$
F =
\begin{cases}
         \varphi_{[s_{l_1}, [s_{l_1 -1},\ldots,[s_{1}, s_2 ]_c \ldots]_c]_c}  &\text{or} \\
         \psi_{[s_{l_2}, [s_{l_2 -1},\ldots,[s_{1}, {s_2}]_c \ldots]_c]_c}    &\text{or} \\
         [ f_{s_l}, [ f_{s_{l-1}},\ldots,[ f_{s_{1}},  f_{s_2}]_c \ldots]_c]_c .
\end{cases}
$$
Since $\hat x_m, \hat w_n$ and $\hat w_s$ are rationally nonzero indecomposable homotopy classes, from the commutative diagram above, we have the following:
\begin{itemize}
\item [{\rm (a)}] $h(\hat x_m ) = \alpha \cdot u_m + D_m (TH_* (X ; \mathbb Q ))$;
\item [{\rm (b)}] $h(\hat w_n ) = \beta_1 \cdot v_n + D_n (TH_* (X ; \mathbb Q ))$, where ${\rm dim}(\hat w_n)$ is odd;
\item [{\rm (c)}] $h(\hat w_n ) = \beta_2 \cdot u_n + D_n (TH_* (X ; \mathbb Q ))$, where ${\rm dim}(\hat w_n)$ is even;
\item [{\rm (d)}] $h(\hat w_s ) = \gamma_1 \cdot v_s + D_s (TH_* (X ; \mathbb Q ))$, where ${\rm dim}(\hat w_s)$ is odd; and
\item [{\rm (e)}] $h(\hat w_s ) = \gamma_2 \cdot u_s + D_s (TH_* (X ; \mathbb Q ))$,  where ${\rm dim}(\hat w_s)$ is even.
\end{itemize}
Here,
\begin{enumerate}
\item all the coefficients $\alpha$, $\beta_1$, $\beta_2$, $\gamma_1$ and $\gamma_2$ are nonzero;
\item $u_i$ and $v_j$ are rational homology generators of the rational homology of $\Omega \Sigma X$ as a tensor algebra
$$
TH_* (X ; \mathbb Q ) \cong T [u_1, u_2, u_3, \ldots, u_i, \ldots ; v_1, v_2, v_3, \ldots, v_j , \ldots]
$$
in which the homology dimensions of $u_i$ and $v_j$ are $2ai$ and $2aj+1$, respectively; and
\item $D_i (TH_* (X ; \mathbb Q ))$ indicates the decomposable elements consisting of the sum of tensor products in the tensor algebra.
\end{enumerate}
Since $\Sigma$ and $\Omega$ are adjoint functors, by Lemma \ref{omega}, we have the following commutative diagram:
\begin{equation} \label{1}
\xymatrix@C=13mm @R=5mm{
[X_{k-1} , \Omega \Sigma X] \ar[r]^-{\cong} \ar[d]^-{i_\sharp}  &[\Sigma (X_{k-1}) , \Sigma X] \ar[d]^-{\Sigma i_\sharp} \ar[r]^-{\cong} &[\Omega \Sigma X_{k-1} , \Omega \Sigma X]_H \ar[d]^-{\Omega \Sigma i_\sharp}   \\
[X , \Omega \Sigma X] \ar[r]^-{\cong}  &[\Sigma X , \Sigma X] \ar[r]^-{\cong} &[\Omega \Sigma X , \Omega \Sigma X]_H ,
} \tag{3.1}
\end{equation}
where $i : X_{k-1} \hookrightarrow X$ is the inclusion.
By using the results above, we first obtain
\begin{equation} \label{2}
\begin{array}{lll}
h \circ ({\Omega \varphi_{[s_{l_1}, [s_{l_1 -1},\ldots,[s_{1}, s_2 ]_c \ldots]_c]_c}}_\sharp) (\hat x_m)
 &= ({\Omega \varphi_{[s_{l_1}, [s_{l_1 -1},\ldots,[s_{1}, s_2 ]_c \ldots]_c]_c}}_*)\circ h (\hat x_m) \\
 &= ({\Omega \varphi_{[s_{l_1}, [s_{l_1 -1},\ldots,[s_{1}, s_2 ]_c \ldots]_c]_c}}_*)
  (\alpha \cdot u_m + D_{m} (TH_* (X ; \mathbb Q )) ).
\end{array}
\tag{3.2}
\end{equation}
Since the restriction of the iterated commutator
$$
\varphi_{[s_{l_1}, [s_{l_1 -1},\ldots,[s_{1}, s_2 ]_c \ldots]_c]_c} |_{(\Sigma X)_{k}} : (\Sigma X)_k \rightarrow \Sigma X
$$
to the $k$-skeleton is inessential, where $k = 2a(s_1 + s_2 +\cdots + s_{l_1})$ (see Remark \ref{Res} with $k = k_1$), by using the commutative diagram \eqref{1} and adjointness, we see that the restriction of the iterated commutators on the loop structure of $\Omega \Sigma X$
$$
\hat \varphi_{[s_{l_1}, [s_{l_1 -1},\ldots,[s_{1}, s_2 ]_c \ldots]_c]_c} |_{X_{k-1}} : X_{k-1} \rightarrow \Omega \Sigma X
$$
is inessential, and thus the restriction
$$
\Omega \varphi_{[s_{l_1}, [s_{l_1 -1},\ldots,[s_{1}, s_2 ]_c \ldots]_c]_c} |_{\Omega \Sigma X_{k-1}} : \Omega \Sigma X_{k-1} \rightarrow \Omega \Sigma X
$$
of the Hopf map
$\Omega \varphi_{[s_{l_1}, [s_{l_1 -1},\ldots,[s_{1}, s_2 ]_c \ldots]_c]_c} : \Omega \Sigma X \rightarrow \Omega \Sigma X$
is also homotopically trivial.

We note that the tensor algebra $TH_* (X ; \mathbb Q )$ has the universal property. In particular, the identity map $1_{H_* (X ; \mathbb Q )}$ on $H_* (X ; \mathbb Q )$ extends uniquely to a homomorphism
$E : H_* (\Omega \Sigma X ; \mathbb Q ) \rightarrow H_* (X ; \mathbb Q )$; that is, the following diagram
\begin{equation} \label{3}
\xymatrix@C=15mm @R=15mm{
H_* (X ; \mathbb Q ) \ar[r]^-{1_{H_*(X ; \mathbb Q )}} \ar[d]_-{\hat\varphi_1}  &H_* (X ; \mathbb Q ) \\
H_* (\Omega \Sigma X ; \mathbb Q ) \ar[ur]_-{E}
} \tag{3.3}
\end{equation}
is strictly commutative (indeed, we also note that $\hat \varphi_1 : X \rightarrow \Omega \Sigma X$ has a left homotopy inverse).
By \eqref{3} and Remark \ref{Res}, we have
\begin{equation} \label{4}
({\Omega \varphi_{[s_{l_1}, [s_{l_1 -1},\ldots,[s_{1}, s_2 ]_c \ldots]_c]_c}}_*)
  (D_{m} (TH_* (X ; \mathbb Q )) ) = 0.
\tag{3.4}
\end{equation}
By \eqref{2}, \eqref{4}, Proposition \ref{Prop1} and Corollary \ref{Cor}, we now have the following:
\begin{equation} \label{5}
\begin{array}{lll}
h \circ ({\Omega \varphi_{[s_{l_1}, [s_{l_1 -1},\ldots,[s_{1}, s_2 ]_c \ldots]_c]_c}}_\sharp) (\hat x_m)
     &= ({\Omega \varphi_{[s_{l_1}, [s_{l_1 -1},\ldots,[s_{1}, s_2 ]_c \ldots]_c]_c}}_*)
         ( \alpha \cdot u_m ) \\
     &= ({\Omega \varphi_{[s_{l_1}, [s_{l_1 -1},\ldots,[s_{1}, s_2 ]_c \ldots]_c]_c}}_*)
        \hat \varphi_{1_*} ( \alpha \cdot u_m ) \\
     &= \hat \varphi_{{[s_{l_1}, [s_{l_1 -1},\ldots,[s_{1}, s_2 ]_c \ldots]_c]_c}_*}
     (\alpha \cdot u_m ) \\
     &= h (\alpha \cdot \hat x_{<s_{l_1}, <s_{l_1 -1},\ldots,<s_{1}, s_2 > \ldots>>} ) .
\end{array}
\tag{3.5}
\end{equation}
Since the homomorphic images of the iterated Samelson products
$$
\hat x_{<s_{l_1}, <s_{l_1 -1},\ldots,<s_{1}, s_2 > \ldots>>} : S^{2am} \rightarrow \Omega \Sigma X,
$$
where $m = s_1 + s_2 + \cdots + s_{l_1}$, are spherical, thus primitive in rational homology, the Cartan-Serre theorem shows that
$$
{\Omega \varphi_{[s_{l_1}, [s_{l_1 -1},\ldots,[s_{1}, s_2 ]_c \ldots]_c]_c}}_\sharp) (\hat x_m)
=  \alpha \cdot \hat x_{<s_{l_1}, <s_{l_1 -1},\ldots,<s_{1}, s_2 > \ldots>>}
$$
as required.

A little elaboration of the argument gives the proofs for the other two cases.
\end{proof}

The Hopf-Thom theorem asserts that the loop space and the suspension of a CW-space are rationally homotopy equivalent to the products of Eilenberg-MacLane spaces and the wedge products of rational spheres, respectively.
In the homotopy category of $G$-spectra, any $(G,c)$-spectrum splits as a product of Eilenberg-MacLane spectra \cite{JG}.
In our case, $\Sigma X$ has the rational homotopy
type of the wedge product of infinite number of spheres; that is,
$$
\Sigma X \simeq_{\mathbb Q} S^{2a+1} \vee S^{2a+2} \vee S^{4a+1} \vee S^{4a+2} \vee \cdots \vee  S^{2an+1} \vee  S^{2an+2}\vee \cdots
$$
for all $n \geq1$.
The following table below (TABLE 3) is a list of various kinds of rationally nonzero indecomposable generators, and pure and hybrid decomposable generators as the basic Whitehead products on the graded rational homotopy group $\pi_{*} (\Sigma X) \otimes \mathbb Q$. We point out that the rational generators of $\pi_{4n+3} (S^{2n+2}_{\mathbb Q})$ were not described on the table for convenience.

\renewcommand{\tabcolsep}{13.48pt}
\renewcommand{\arraystretch}{1.0}
\begin{table}[h!]
\begin{center}
 \caption{Generators on $\pi_{*} (\Sigma X) \otimes \mathbb Q$ when $a=1$}
    \label{tab:table3}
    \begin{tabular}{ |>{\columncolor{lightgray}} p{1.3cm} | p{2.2cm} | p{3.3cm}  | p{3.8cm} |}
    \specialrule{1pt}{0.01pt}{0.01pt}
     Dimension & Indecomposables & Pure Decomposables & Hybrid Decomposables \\ \specialrule{1pt}{0.01pt}{0.01pt}
      $3$ & $\chi_1$ & - & -  \\    \specialrule{0.1pt}{0.01pt}{0.01pt}
      $4$ & $\eta_1$ & - & -  \\    \specialrule{0.1pt}{0.01pt}{0.01pt}
      $5$ & $\chi_2$ & - & -  \\    \specialrule{0.1pt}{0.01pt}{0.01pt}
      $6$ & $\eta_2$ & - &$[\chi_1, \eta_1]$  \\    \specialrule{0.1pt}{0.01pt}{0.01pt}
      $7$ & $\chi_3$ & $[\chi_1, \chi_2]$ & -  \\    \specialrule{0.1pt}{0.01pt}{0.01pt}
      $8$ & $\eta_3$ & - & $[\chi_1, \eta_2]$, $[\eta_1,\chi_2]$, \newline $[\chi_1, [\chi_1,\eta_1]]$ \\    \specialrule{0.1pt}{0.01pt}{0.01pt}
      $9$ & $\chi_4$ & $[\chi_1, \chi_3]$, $[\chi_1, [\chi_1,\chi_2]]$, $[\eta_1, \eta_2]$ & $[\eta_1, [\chi_1,\eta_1]]$ \\
                                                                                                             \specialrule{0.1pt}{0.01pt}{0.01pt}
      $10$ & $\eta_4$ & - & $[\chi_1, \eta_3]$, $[\chi_2, \eta_2]$, \newline $[\eta_1, \chi_3]$, $[\chi_2, [\chi_1, \eta_1]]$, \newline $[\eta_1, [\chi_1, \chi_2]]$, $[\chi_1, [\chi_1, \eta_2]]$, $[\chi_1, [\chi_1, [\chi_1,\eta_1]]]$   \\    \specialrule{0.1pt}{0.01pt}{0.01pt}
      $11$ & $\chi_5$ & $[\chi_1,\chi_4]$, $[\eta_1, \eta_3]$, \newline  $[\chi_2, \chi_3]$, $[\chi_2, [\chi_1, \chi_2]]$, \newline  $[\chi_1,[\chi_1, \chi_3]]$, $[\chi_1,[\chi_1, [\chi_1,\chi_2]]]$ & $[\eta_1, [\chi_1, \eta_2]]$, \newline $[\eta_1, [\eta_1,\chi_2]]$, $[\eta_1, [\chi_1, [\chi_1,\eta_1]]]$  \\    \specialrule{0.1pt}{0.01pt}{0.01pt}
      $12$ & $\eta_5$ & $[\eta_1,[\eta_1, \eta_2]]$ &
      $[\chi_1, \eta_4]$, $[\chi_1, [\chi_1, \eta_3]]]$, \newline
      $[\chi_1, [\chi_1, [\chi_1, \eta_2]]]$,
      $[\chi_1, [\chi_1, [\chi_1, [\chi_1,\eta_1]]]]$,
      $[\eta_1, \chi_4]$, $[\eta_1, [\chi_1, \chi_3]]$, \newline
      $[\eta_1, [\chi_1, [\chi_1,\chi_2]]]$,
      $[\eta_1, [\eta_1, [\chi_1,\eta_1]]]$,
      $[\chi_2, \eta_3]$, \newline $[\chi_2, [\chi_1, \eta_2]]$,
      $[\chi_2,[\eta_1,\chi_2]]$, $[\chi_2, [\chi_1, [\chi_1,\eta_1]]]$, $[\eta_2, \chi_3]$, $[\eta_2, [\chi_1, \chi_2]]$, \newline $[[\chi_1, \eta_1], [\chi_1, \chi_2]]$ \\ \specialrule{0.1pt}{0.01pt}{0.01pt}
      $13$ & $\chi_6$ & $[\chi_1, \chi_5]$, $[\chi_1, [\chi_1,\chi_4]]$, \newline
      $[\chi_1, [\chi_1, [\chi_1, [\chi_1,\chi_2]]]]$, $[\chi_1, [\chi_1,[\chi_1, \chi_3]]]$, $[\eta_1,\eta_4]$, $[\chi_2,\chi_4]$, \newline $[\chi_2,[\chi_1, \chi_3]]$, $[\chi_2,[\chi_1, [\chi_1,\chi_2]]]$, $[\eta_2,\eta_3]$
      &$[\eta_1,[\chi_1, \eta_3]]$, \newline
      $[\eta_1,[\eta_1, \chi_3]]$, $[\eta_1,[\chi_1, [\chi_1, \eta_2]]]$,  $[\eta_1,[\chi_1, [\chi_1, [\chi_1,\eta_1]]]]$, $[\eta_1,[\eta_1, [\chi_1, \chi_2]]]$, $[\chi_2,[\eta_1, \eta_2]]$, $[\chi_2,[\eta_1, [\chi_1,\eta_1]]]$, $[\eta_2,[\chi_1, \eta_2]]$, $[\eta_2,[\eta_1,\chi_2]]$, $[\eta_2,[\chi_1, [\chi_1,\eta_1]]]$,
      $[[\chi_1, \eta_1],[\chi_1, \eta_2]]$, $[[\chi_1, \eta_1],[\eta_1,\chi_2]]$, $[[\chi_1, \eta_1],[\chi_1, [\chi_1,\eta_1]]]$\\ \specialrule{0.1pt}{0.01pt}{0.01pt}
      $\vdots$ & $\vdots$ & $\vdots$ & $\vdots$  \\ \specialrule{0.1pt}{0.01pt}{0.01pt}
      $2n+1$ & $\chi_n$ & $[\chi_1, \chi_{n-1}]$, $\ldots$ & $\ldots$  \\ \specialrule{0.1pt}{0.01pt}{0.01pt}
      $2n+2$ & $\eta_n$ & $\ldots$ & $[\chi_1, \eta_{n-1}]$, $\ldots$  \\ \specialrule{0.1pt}{0.01pt}{0.01pt}
      $\vdots$ & $\vdots$ & $\vdots$ & $\vdots$  \\
    \specialrule{1pt}{0.01pt}{0.01pt}
    \end{tabular}
\end{center}
\end{table}
\vspace{0.4cm}

\begin{rmk}
We observe that there are many different types of basic Whitehead products whenever
the homotopy dimensions are on the increase. Take the basic Whitehead products $[[\chi_1, \eta_1], [\chi_1, \chi_2]]$ in dimension $12$, and $[[\chi_1, \eta_1],[\chi_1, \eta_2]]$, $[[\chi_1, \eta_1],[\eta_1,\chi_2]]$, $[[\chi_1, \eta_1],[\chi_1, [\chi_1,\eta_1]]]$ in dimension $13$ as examples.
By ignoring (or killing) all torsion subgroups of the homotopy groups and considering the homotopy groups rationally, we have
$$
{\rm rank}_{\mathbb{Z}} \pi_{*} (\Sigma X)/{\rm torsion} = {\rm dim}_{\mathbb{Q}} \pi_{*} (\Sigma X) \otimes \mathbb Q;
$$
that is, the number of generators on the homotopy groups modulo torsion subgroups is exactly the same as the dimension of the rational vector spaces.
Therefore, one can construct the corresponding indecomposable generators, and pure and hybrid decomposable elements
on the graded homotopy groups modulo torsion subgroups $\pi_{*}(\Sigma X)/{\rm torsion}$.
\end{rmk}

\medskip

\section{The main theorem}

By using all the ingredients in Sections \ref{section2} and \ref{section3}, we now describe the main result of this paper and prove it as follows:

\begin{thm}\label{thm}
Let $X: = K( {\mathbb Z}, 2a) \vee \bigvee_{j=1}^{\infty} K( {\mathbb Z}, 2aj+1)$ be the wedge product of the
Eilenberg-MacLane spaces, where `$a$' is a positive integer. Then $SNT(\Sigma X) = \{ [\Sigma X] \}$.
\end{thm}

Let $r : {\mathbb Z_+} \rightarrow {\mathbb Z_+} (m \mapsto r_m)$ be a sequence of positive integers defined by
$$
r_m =
\begin{cases}
         2an+1 &\text{if}~~ m = 2n-1 \\
         2an+2 &\text{if}~~ m = 2n
\end{cases}
$$
for all $n \geq 1$.

\begin{proof}
It can be shown that the graded homotopy groups modulo torsion subgroups
$$
qL = \pi_{*}(\Sigma X)/{\rm torsion}
$$
and
$$
qL_{\leq r_m} = \pi_{\leq r_m}(\Sigma X)/{\rm torsion}
$$
have the structures of the quasi-Lie algebras whose brackets are the Whitehead products corresponding to $\mathcal L$ and
$\mathcal L_{\leq r_m}$ as the rational homotopy counterparts, respectively.

Let $I_{r_m} (qL)$ and $D_{r_m} (qL)$ be the indecomposable and (pure and hybrid) decomposable components, respectively, of the graded homotopy group modulo
torsion subgroups $\pi_{r_m}(\Sigma X)/{\rm torsion}$. Then, by using the adjointness and the dual of the Samelson products in Theorem \ref{thm4}, we can consider the following maps
$$
{\varphi_{[s_{l_1}, [s_{l_1 -1},\ldots,[s_{1}, s_2 ]_c \ldots]_c]_c}}_\sharp : I_{r_{m_1}} (qL) \rightarrow
D_{r_{m_1}} (qL),
$$
$$
{\psi_{[s_{l_2}, [s_{l_2 -1},\ldots,[s_{1}, {s_2}]_c \ldots]_c]_c}}_\sharp : I_{r_{m_2}} (qL) \rightarrow
D_{r_{m_2}} (qL)
$$
and
$$
{[ f_{s_l}, [ f_{s_{l-1}},\ldots,[ f_{s_{1}},  f_{s_2}]_c \ldots]_c]_c}_\sharp : I_{r_{m_3}} (qL) \rightarrow
D_{r_{m_3}} (qL),
$$
where
\begin{itemize}
\item [{\rm (1)}] $r_{m_1} = 2a(s_1 + s_2 +\cdots + s_{l_1})+1$;
\item [{\rm (2)}] $r_{m_2} = 2a(s_1 + s_2 +\cdots + s_{l_2}) + l_2+1$; and
\item [{\rm (3)}] $r_{m_3} = 2a(s_1 + s_2 +\cdots + s_l) + l-t+1$.
\end{itemize}

For each iterated Samelson product or for each iterated Whitehead product as the dual object, by Theorem \ref{thm4} again, we can guarantee the existence of the iterated commutator which corresponds to the given iterated Samelson product satisfying that the type of the iterated Samelson product is exactly the same as the type of the loop of the iterated commutator. By delooping the iterated commutators and considering the adjointness $(\Sigma, \Omega)$, we have the following correspondence or relationship (see Table 4) between elements of graded homotopy groups modulo torsion subgroups $\pi_{*}(\Sigma X)/{\rm torsion}$ and (pure and hybrid) iterated commutators of self-maps on $\Sigma X$.
\renewcommand{\tabcolsep}{13.48pt}
\renewcommand{\arraystretch}{1.0}
\begin{table}[h!]
\begin{center}
 \caption{The correspondence of elements of homotopy groups and iterated commutators}
    \label{tab:table3}
    \begin{tabular}{ |p{5.3cm} |  p{5.3cm} |}
    \specialrule{1pt}{0.01pt}{0.01pt}
      \cellcolor{lightgray} Elements of $\pi_{*}(\Sigma X)/{\rm torsion}$ & \cellcolor{lightgray}Iterated commutators \\ \specialrule{1pt}{0.01pt}{0.01pt}
      $x_{i_1}$ & $\varphi_{i_1}$  \\    \specialrule{0.1pt}{0.01pt}{0.01pt}
      $y_{j_1}$ & $\psi_{j_1}$  \\    \specialrule{0.1pt}{0.01pt}{0.01pt}
      $[x_{i_1}, x_{i_2}]$  & $[\varphi_{i_1}, \varphi_{i_2}]$  \\    \specialrule{0.1pt}{0.01pt}{0.01pt}
      $[y_{j_1}, y_{j_2}]$  &$[\psi_{j_1}, \psi_{j_2}]$  \\    \specialrule{0.1pt}{0.01pt}{0.01pt}
      $\vdots$  &$\vdots$  \\    \specialrule{0.1pt}{0.01pt}{0.01pt}
      $[x_{i_s}, y_{j_t}]$ & $[\varphi_{i_s}, \psi_{j_t}]$   \\    \specialrule{0.1pt}{0.01pt}{0.01pt}
      $[x_{i_s},[x_{i_s}, y_{j_t}]]$ & $[\varphi_{i_s},[\varphi_{i_s}, \psi_{j_t}]]$ \\    \specialrule{0.1pt}{0.01pt}{0.01pt}
      $\vdots$ & $\vdots$  \\ \specialrule{0.1pt}{0.01pt}{0.01pt}
      $[[x_{i_s}, x_{i_t}], [y_{j_s}, y_{j_t}]]$ & $[[\varphi_{i_s}, \varphi_{i_t}], [\psi_{j_s}, \psi_{j_t}]]$ \\ \specialrule{0.1pt}{0.01pt}{0.01pt}
      $\vdots$ & $\vdots$  \\ \specialrule{0.1pt}{0.01pt}{0.01pt}
      $[[x_{i_s}, x_{i_t}], [y_{j_u}, [y_{j_s}, y_{j_t}]]]$ & $[[\varphi_{i_s}, \varphi_{i_t}], [\psi_{j_u},[\psi_{j_s}, \psi_{j_t}]]]$ \\ \specialrule{0.1pt}{0.01pt}{0.01pt}
      $\vdots$ & $\vdots$  \\ \specialrule{0.1pt}{0.01pt}{0.01pt}
      $[x_{s_{l_1}}, [x_{s_{l_1 -1}},\ldots,[x_{s_{1}}, x_{s_2}] \ldots]]$  & $[\varphi_{s_{l_1}}, [ \varphi_{s_{l_1 -1}},\ldots,[ \varphi_{s_{1}},  \varphi_{s_2}]_c \ldots]_c]_c$ \\ \specialrule{0.1pt}{0.01pt}{0.01pt}
      $[y_{s_{l_2}}, [y_{s_{l_2 -1}},\ldots,[y_{s_{1}}, y_{s_2}] \ldots]]$  & $[\psi_{s_{l_2}}, [ \psi_{s_{l_2 -1}},\ldots,[ \psi_{s_{1}},  \psi_{s_2}]_c \ldots]_c]_c$ \\ \specialrule{0.1pt}{0.01pt}{0.01pt}
      $[z_{s_l}, [z_{s_{l-1}},\ldots,[z_{s_{1}}, z_{s_2}] \ldots]]$  & $[f_{s_l}, [ f_{s_{l-1}},\ldots,[ f_{s_{1}},  f_{s_2}]_c \ldots]_c]_c$ \\  \specialrule{0.1pt}{0.01pt}{0.01pt}
      $\vdots$ & $\vdots$   \\ \specialrule{0.1pt}{0.01pt}{0.01pt}                                                                                                             \specialrule{1pt}{0.01pt}{0.01pt}
    \end{tabular}
\end{center}
\end{table}
\vspace{0.4cm}

Let
$$
\pi : qL_{\leq r_m} \rightarrow I_{r_m} (qL)
$$
and
$$
i : D_{r_m} (qL)) \hookrightarrow qL_{\leq r_m}
$$
be the projection and the inclusion, respectively. Then by considering the indecomposable and (pure and hybrid) decomposable generators (see Table 3 when $a = 1$), we can see that the following sequence
\begin{equation} \label{4.1}
\xymatrix@C=4.5mm @R=0.5mm{
0 \ar[r]^-{} &{\rm Hom}(I_{r_m} (qL), D_{r_m} (qL)) \ar[r]^-{P} &{\rm Aut}(qL_{\leq r_m})
 \ar[r]^-{R} &{\rm Aut} (qL_{\leq r_{m-1}}) \oplus {\rm Aut} (I_{r_m} (qL)) \ar[r]^-{} &0
}
\tag{4.1}
\end{equation}
is exact for each $m \geq 1$ (see \cite[page 289]{MM}).
Here,
\begin{enumerate}
\item[{\rm (1)}] $P(-) = 1 + i\circ (-)\circ \pi$; that is,
$$
P({\varphi_{[s_{l_1}, [s_{l_1 -1},\ldots,[s_{1}, s_2 ]_c \ldots]_c]_c}}_\sharp )
= 1 + i\circ ({\varphi_{[s_{l_1}, [s_{l_1 -1},\ldots,[s_{1}, s_2 ]_c \ldots]_c]_c}}_\sharp)\circ \pi ,
$$
$$
P({\psi_{[s_{l_2}, [s_{l_2 -1},\ldots,[s_{1}, {s_2}]_c \ldots]_c]_c}}_\sharp )
= 1 + i\circ ({\psi_{[s_{l_2}, [s_{l_2 -1},\ldots,[s_{1}, {s_2}]_c \ldots]_c]_c}}_\sharp)\circ \pi
$$
and
$$
P({[ f_{s_l}, [ f_{s_{l-1}},\ldots,[ f_{s_{1}},  f_{s_2}]_c \ldots]_c]_c}_\sharp ) = 1 + i\circ ({[ f_{s_l}, [ f_{s_{l-1}},\ldots,[ f_{s_{1}},  f_{s_2}]_c \ldots]_c]_c}_\sharp)\circ \pi ;
$$ and
\item[{\rm (2)}] $R$ is given by the restriction and the projection of
the group of automorphisms of $qL_{\leq r_m}$ to $qL_{\leq r_{m-1}}$ and into $I_{r_m} (qL)$, respectively.
\end{enumerate}
We remark that there is only one indecomposable generator (if any) in each dimension in the quasi-Lie algebra
$$
qL = \pi_{*}(\Sigma X)/{\rm torsion}.
$$
Thus, we have
$$
I_{r_m} (qL) \cong \mathbb Z
$$
for each $m \geq 1$. Moreover, by considering the indecomposable elements, and pure and hybrid decomposable elements of
$\pi_* (\Sigma X)/{\rm torsion}$, we obtain
\begin{enumerate}
\item[{\rm (1)}] ${\rm Aut} (\pi_{2a+1} (\Sigma X)/{\rm torsion}) \cong {\mathbb Z}_2$;
\item[{\rm (2)}] ${\rm Aut} (\pi_{2a+2} (\Sigma X)/{\rm torsion}) \cong {\mathbb Z}_2$;
\item[{\rm (3)}] ${\rm Aut} (\pi_{4a+1} (\Sigma X)/{\rm torsion}) \cong {\mathbb Z}_2$;
\item[{\rm (4)}] ${\rm Aut} (\pi_{\leq 4a+1} (\Sigma Y)/{\rm torsion}) \cong {\mathbb Z}_2 \oplus {\mathbb Z}_2 \oplus {\mathbb Z}_2$; and
\item[{\rm (5)}] ${\rm Aut} (\pi_{\leq r_m} (\Sigma X)/{\rm torsion})$ is both nonabelian and infinite for all $m \geq 4$.
\end{enumerate}
In order to show that the set of all same $n$-types of $\Sigma X$ is the set consisting of only one element as a
single homotopy type of itself in the homotopical same $n$-type point of view, we argue by induction on $m$.
We first consider the map
$$
{\rm Aut}(\Sigma X) \rightarrow {\rm Aut}(qL_{r_1}) = {\rm Aut}(qL_{2a\cdot 1+1}) \cong {\mathbb Z}_2
$$
so that the induction step begins. We secondly suppose that the map
$$
{\rm Aut}(\Sigma X) \rightarrow {\rm Aut}(qL_{\leq r_{m-1}})
$$
has a finite cokernel, and then show that
$$
{\rm Aut}(\Sigma X) \rightarrow {\rm Aut}(qL_{\leq r_{m}})
$$
has a finite index for each $m$. For any kinds of iterated Whitehead products
$$
x_{[s_{l_1}, [s_{l_1 -1},\ldots,[s_{1}, s_2] \ldots]]} \in D_{2a(s_1 + s_2 +\cdots + s_{l_1})+1}(qL) ,
$$
$$
y_{[s_{l_2}, [s_{l_2 -1},\ldots,[s_1, s_2] \ldots]]} \in D_{2a(s_1 + s_2 +\cdots + s_{l_2}) +l_2 +1}(qL)
$$
and
$$
[z_{s_l}, [z_{s_{l-1}},\ldots,[z_{s_{1}}, z_{s_2}] \ldots]] \in D_{2a(s_1 + s_2 +\cdots + s_l) + l-t+1}(qL)
$$
there are (pure and hybrid) iterated commutators
$$
\varphi_{[s_{l_1}, [s_{l_1 -1},\ldots,[s_{1}, s_2 ]_c \ldots]_c]_c} : \Sigma X \rightarrow \Sigma X,
$$
$$
\psi_{[s_{l_2}, [s_{l_2 -1},\ldots,[s_{1}, {s_2}]_c \ldots]_c]_c} : \Sigma X \rightarrow \Sigma X
$$
and
$$
[f_{s_l}, [ f_{s_{l-1}},\ldots,[ f_{s_{1}},  f_{s_2}]_c \ldots]_c]_c : \Sigma X \rightarrow \Sigma X
$$
corresponding to the the types of iterated Whitehead products, respectively. We note that the types of iterated Whitehead products are exactly the same as the types of (pure and hybrid) iterated commutators.
By considering the Whitehead theorem and by using the results in Sections 2 and 3, we now have the homotopy self-equivalences
$$
1+\varphi_{[s_{l_1}, [s_{l_1 -1},\ldots,[s_{1}, s_2 ]_c \ldots]_c]_c} : \Sigma X \rightarrow \Sigma X,
$$
$$
1+\psi_{[s_{l_2}, [s_{l_2 -1},\ldots,[s_{1}, {s_2}]_c \ldots]_c]_c} : \Sigma X \rightarrow \Sigma X
$$
and
$$
1+[f_{s_l}, [ f_{s_{l-1}},\ldots,[ f_{s_{1}},  f_{s_2}]_c \ldots]_c]_c : \Sigma X \rightarrow \Sigma X,
$$
such that the restrictions to the quasi Lie subalgebras $qL_{\leq r_{m_1 -1}}$, $qL_{\leq r_{m_2 -1}}$ and $qL_{\leq r_{m -1}}$, respectively, are the identity maps; that is,
$$
(1+\varphi_{[s_{l_1}, [s_{l_1 -1},\ldots,[s_{1}, s_2 ]_c \ldots]_c]_c} )_{\sharp} |_{qL_{\leq r_{m_1 -1}}}  =  1_{qL_{\leq r_{m_1 -1}}},
$$
$$
(1+\psi_{[s_{l_2}, [s_{l_2 -1},\ldots,[s_{1}, {s_2}]_c \ldots]_c]_c})_{\sharp} |_{qL_{\leq r_{m_2 -1}}}  =  1_{qL_{\leq r_{m_2 -1}}}
$$
and
$$
(1+[f_{s_l}, [ f_{s_{l-1}},\ldots,[ f_{s_{1}},  f_{s_2}]_c \ldots]_c]_c )_{\sharp} |_{qL_{\leq r_{m -1}}}  =  1_{qL_{\leq r_{m -1}}}.
$$
By using Proposition \ref{Re} and by delooping and taking the adjointness of the iterated Samelson products in Theorem \ref{thm4}, we have
$$
(1+ \varphi_{[s_{l_1}, [s_{l_1 -1},\ldots,[s_{1}, s_2 ]_c \ldots]_c]_c} )_{\sharp}(x_m)  =
x_m + \alpha \cdot  x_{[s_{l_1}, [s_{l_1 -1},\ldots,[s_{1}, s_2] \ldots]]},
$$
$$
(1+\psi_{[s_{l_2}, [s_{l_2 -1},\ldots,[s_{1}, {s_2}]_c \ldots]_c]_c})_{\sharp} (w_n)
= w_n + \beta \cdot y_{[s_{l_2}, [s_{l_2 -1},\ldots,[s_1, s_2] \ldots]]}
$$
and
$$
(1+[f_{s_l}, [ f_{s_{l-1}},\ldots,[ f_{s_{1}},  f_{s_2}]_c \ldots]_c]_c )_{\sharp} (w_s)  =  w_s + \gamma \cdot [z_{s_l}, [z_{s_{l-1}},\ldots,[z_{s_{1}}, z_{s_2}] \ldots]].
$$
We should be careful when applying the dimensions of the quasi-Lie algebra to induction; if
$$
r_{m-1} =
\begin{cases}
         2a(s_1+ s_2 + \cdots + s_{l_1 -1} ) +1 &\text{or} \\
         2a(s_1+ s_2 + \cdots + s_{l_2 -1} ) +l_2 &\text{or} \\
         2a(s_1+ s_2 + \cdots + s_{l-1} ) +l-t, \\
\end{cases}
$$
where $1\leq t \leq l-1$, then the next term $r_m$ is determined by the next dimension of the one of them depending on $a, l_1, l_2$ and $l$.
Finally, by considering the indecomposable and (pure and hybrid) decomposable generators, induction hypothesis and Theorem \ref{MM},
we finally complete the proof of the main theorem of this paper.
\end{proof}

\begin{rmk} We observe that there may be a (pure or hybrid) decomposable homotopy class of $\pi_{*}(\Sigma X)/{\rm torsion}$ ($a \geq 2$) without an indecomposable homotopy class in a certain dimension. More precisely, if $a=2$, $y_1 \in \pi_{6}(\Sigma X)/{\rm torsion}$ and $y_2 \in \pi_{10}(\Sigma X)/{\rm torsion}$, then it can be seen that $[y_1, y_2]$ is an element of $\pi_{15}(\Sigma X)/{\rm torsion}$ which is a basic Whitehead product as a rationally nonzero homotopy class, and that there is no indecomposable generator in this dimension. However, there is no difficulty in considering the mathematical induction because the epimorphism of (4-1) becomes an isomorphism because $I_{r_m} (qL) =0$.
\end{rmk}

For two nilpotent spaces, $Y$ and $Z$, the space $Y$ is said to be a {\it clone} of $Z$ if $Y$ and $Z$ have the same homotopy $n$-type for all $n$ and the localizations $Y_{(p)}$ and $Z_{(p)}$ at each prime $p$ are homotopy equivalent. The main result in this paper says that the set of all clones of $\Sigma X$ is the set consisting of the singe homotopy type.

The reader may have noticed that the $k$th suspensions $\Sigma^k X , k \geq 2$ were not described in this paper. Indeed, the homotopy self-equivalences
$1 + [{\rm commutators}] \in {\rm Aut}(\Sigma X)$
considered in this paper are not as well behaved as we might wish on the self-maps of the $k$th suspension of a given CW-complex $X$ because the group $[\Sigma^k X , \Sigma^k X]$ becomes abelian for $k \geq 2$.
As a matter of fact, the commutators are all vanishing in this abelian group.
However, it is natural for us to ask that there are many kinds of rationally nontrivial homotopy classes $[f]$ of self-maps on the abelian group  $[\Sigma^k X , \Sigma^k X], k\geq 2$ such that
$\Sigma f : [\Sigma^{k+1} X , \Sigma^{k+1} X]$ is inessential.
We end this paper with the following question:

\medskip

\noindent {\bf Question 4.3.}
Let $\Bbb G$ be a finitely generated abelian group.
Can we calculate $SNT(\Sigma^k X)$ for $k\geq 2$, where $X$ is the Eilenberg-MacLane space of type $({\Bbb G},n), n \geq 1$ or the wedge product of the Eilenberg-MacLane spaces of various types?

\end{document}